\documentclass[10pt,reqno]{amsart}
\usepackage{amssymb,mathrsfs,graphicx}
\usepackage{ifthen}
\usepackage[hidelinks]{hyperref}
\usepackage[margin=1in]{geometry}
\usepackage{caption}
\usepackage{sidecap}
\usepackage{rotating}
\usepackage{enumitem}

\usepackage{cancel}

\usepackage{colortbl}
\definecolor{black}{rgb}{0.0, 0.0, 0.0}
\definecolor{red}{rgb}{1.0, 0.5, 0.5}
\definecolor{RevisionRed}{rgb}{1,0,0}

\provideboolean{shownotes} 
\setboolean{shownotes}{true} 
\newcommand{\margnote}[1]{
\ifthenelse{\boolean{shownotes}}%
{\marginpar{\raggedright\tiny\texttt{#1}}}%
{}%
}
\newcommand{\hole}[1]{
\ifthenelse{\boolean{shownotes}}%
{\begin{center} \fbox{ \rule {.25cm}{0cm} \rule[-.1cm]{0cm}{.4cm}
\parbox{.85\textwidth}{\begin{center} \texttt{#1}\end{center}} \rule
{.25cm}{0cm}}\end{center}} {} }

\title[Global existence for an asymmetric bipolar Euler--Poisson system]
{Global existence and time decay for a bipolar Euler--Poisson system
with one pressureless and undamped fluid}

\author[Choi]{Young-Pil Choi}
\address{Department of Mathematics, Yonsei University, Seoul 03722, Republic of Korea}
\email{ypchoi@yonsei.ac.kr}

 \author[Tang]{Houzhi Tang}
\address{School of Mathematics and Statistics, Anhui Normal University, Wuhu 241002, P.R. China}
	\email{houzhitang@ahnu.edu.cn}

\author[Zou]{Weiyuan Zou}
\address{College of Mathematics and Physics, Beijing University of Chemical Technology, Beijing 100029, P.R. China}
\email{zwy@amss.ac.cn}

\numberwithin{equation}{section}

\newtheorem{thm}{Theorem}[section]
\newtheorem{lem}[thm]{Lemma}
\newtheorem{prop}[thm]{Proposition}
\newtheorem{rem}{Remark}[section]

\makeatletter
\def\moverlay{\mathpalette\mov@rlay}
\def\mov@rlay#1#2{\leavevmode\vtop{%
   \baselineskip\z@skip \lineskiplimit-\maxdimen
   \ialign{\hfil$\m@th#1##$\hfil\cr#2\crcr}}}
\newcommand{\charfusion}[3][\mathord]{
    #1{\ifx#1\mathop\vphantom{#2}\fi
        \mathpalette\mov@rlay{#2\cr#3}
      }
    \ifx#1\mathop\expandafter\displaylimits\fi}
\makeatother

\begin{document}
\allowdisplaybreaks

\date{\today}



\begin{abstract}
We study the Cauchy problem for a three-dimensional bipolar Euler--Poisson system in which one fluid is pressureless and undamped, while the other is subject to momentum relaxation. For sufficiently small smooth perturbations of a constant equilibrium, we prove the global existence and uniqueness of smooth solutions under an irrotationality assumption on the initial velocity of the pressureless fluid, together with algebraic time-decay estimates. The main difficulty is that the velocity of the pressureless fluid is dissipated only indirectly through the Poisson coupling, and this mechanism degenerates strongly at high frequencies, leading to a regularity-loss structure. We overcome this difficulty by combining refined Green-function estimates, a low--middle--high frequency decomposition, and high-order nonlinear energy estimates adapted to the asymmetric regularity hierarchy. The result establishes a global small-data theory for this asymmetric regime, in which pressure and damping are simultaneously absent from the same fluid.
\end{abstract}

\maketitle  

\tableofcontents

%
%
%
%

\section{Introduction}

%
%
%
%
%
%

\subsection{The asymmetric model and its motivation}

Euler--Poisson systems arise naturally as hydrodynamic descriptions of charged particles in semiconductor devices and plasmas; see \cite{MRS90}. In a two-carrier description, two charged species evolve according to their own transport dynamics while interacting through a common self-consistent electric field. At the level of a general hydrodynamic model, the momentum equation for each carrier may contain both a pressure contribution and a momentum-relaxation term, schematically of the form
\[
\partial_t(\rho_i u_i) +\operatorname{div}(\rho_i u_i\otimes u_i) +\theta_i\nabla P_i(\rho_i) =\zeta_i\rho_i\nabla\Phi-\nu_i\rho_i u_i.
\]
Here the coefficients $\theta_i$ and $\nu_i$ reflect, after nondimensionalization, thermal or acoustic scales and momentum-relaxation rates of the corresponding species, while $\zeta_i$ represents the sign of the charge. These physical parameters need not be comparable for the two carriers. Regimes in which one species is much colder or relaxes much more weakly than the other naturally lead to strongly asymmetric reduced models.

The system studied in this paper may be viewed as a degenerate limiting case of such a two-carrier Euler--Poisson model. One carrier is taken to be both pressureless and undamped, whereas the second carrier retains pressure and momentum relaxation. More precisely, we consider the following three-dimensional bipolar Euler--Poisson system:
\begin{equation}\label{main1}
\left\{
\begin{aligned}
&\partial_t \rho_1+\operatorname{div}(\rho_1u_1)=0,\\
&\partial_t(\rho_1u_1)+\operatorname{div}(\rho_1u_1\otimes u_1) =\rho_1\nabla\Phi,\\
&\partial_t \rho_2+\operatorname{div}(\rho_2u_2)=0,\\
&\partial_t(\rho_2u_2)+\operatorname{div}(\rho_2u_2\otimes u_2) +\nabla P(\rho_2) =-\rho_2\nabla\Phi-\rho_2u_2,\\
&\Delta\Phi=\rho_1-\rho_2.
\end{aligned}
\right.
\end{equation}
Here $\rho_i=\rho_i(t,x)>0$ and $u_i=u_i(t,x)\in\mathbb R^3$ $(i=1,2)$ denote the densities and velocities, respectively, and $\Phi=\Phi(t,x)$ is the electric potential. The pressure law is given by
\[
P(\rho_2)=\rho_2^\gamma,\quad \gamma>1.
\]
The Cauchy problem is supplemented with the initial data
\[
(\rho_1,u_1,\rho_2,u_2)|_{t=0} =(\rho_{10},u_{10},\rho_{20},u_{20})(x)
\]
and the far-field condition
\[
\lim_{|x|\to\infty}(\rho_1,u_1,\rho_2,u_2)(x,t) =(\rho_*,0,\rho_*,0).
\]
The electric field is determined by
\[
\nabla\Phi=-\nabla(-\Delta)^{-1}(\rho_1-\rho_2).
\]
Only $\nabla\Phi$ enters the equations, so the choice of the additive constant in $\Phi$ does not affect the system. Throughout the paper, we normalize the equilibrium density to $\rho_*=1$.

From the physical point of view, \eqref{main1} represents a mixed thermal and relaxation regime in which the first species is effectively cold and has a vanishing relaxation rate, while the second species remains pressurized and dissipative. Our purpose is not to derive this regime from a particular microscopic scaling, but rather to study the dynamics in this degenerate regime of the bipolar Euler--Poisson system. This degeneracy has important mathematical consequences. When pressure and relaxation act on both components, they provide stabilizing mechanisms for both fluids. Once both mechanisms are removed from one component, the dissipative symmetry between the two species is broken, and dissipation of the first fluid can only be transmitted indirectly through the Poisson coupling. As shown below, this indirect transfer becomes highly degenerate at high frequencies.

%
%
%
%
%
%
\subsection{Related results}

We next recall some related results. A classical bipolar hydrodynamic model contains pressure and momentum relaxation in both momentum equations. Global smooth solvability near a constant equilibrium state in several space dimensions was established in \cite{AJ03}. The large-time behavior of bipolar hydrodynamic models was studied in \cite{GHL03,HMW11}. Global existence and long-time behavior for the two-fluid Euler--Maxwell system, together with analogous results for the two-fluid Euler--Poisson system, were obtained in \cite{Pen12}. Global existence and $L^2$-decay for the three-dimensional bipolar Euler--Poisson system were proved by combining Green-function analysis with energy estimates in \cite{LY12}. Refined decay estimates based on energy and interpolation arguments involving negative Sobolev or Besov norms were obtained in \cite{WW14}. Global well-posedness in critical regularity spaces was studied in \cite{PX13}. Optimal decay estimates in Besov spaces for the Euler--Poisson two-fluid system were established in \cite{XK15}, while pointwise $L^p$ estimates were obtained in \cite{WL16}. These results demonstrate the strong stabilizing effect of pressure and momentum relaxation when they act on both carrier velocities.

Closer to the asymmetric setting considered here, the global existence of small smooth solutions for a partially damped two-fluid Euler--Poisson system on $\mathbb T^3$ was established in \cite{ZZ19}, under a high-regularity assumption and an irrotationality condition on the undamped velocity. In that model, pressure is retained in both momentum equations, while damping acts on only one of the two velocities. The present system is more degenerate in that the undamped component is also pressureless. Moreover, our problem is posed on $\mathbb R^3$, where the analysis must capture not only global regularity but also large-time decay and the frequency-dependent transfer of dissipation.

Closer to the asymmetric setting considered here, the global existence of small smooth solutions for a partially damped two-fluid Euler--Poisson system on $\mathbb T^3$ was established in \cite{ZZ19}, under a high-regularity assumption and an irrotationality condition on the undamped velocity. In that model, pressure is retained in both momentum equations, while damping acts on only one of the two velocities. The present system is more degenerate in that the undamped component is also pressureless. Moreover, our problem is posed on $\mathbb R^3$, where the analysis must capture not only global regularity but also large-time decay and the frequency-dependent transfer of dissipation. A related asymmetric dissipation mechanism was studied for a partially damped two-fluid Euler--Maxwell system on $\mathbb T^3$ in \cite{LZ23}, where global small-data well-posedness and time-decay estimates were established. Pressureless two-fluid Euler--Poisson systems have also been considered from a different perspective: the long-wave dispersive limit in $\mathbb R^3$ toward the Zakharov--Kuznetsov equation was justified in \cite{Li21}.

Related results are also available for one-fluid Euler--Poisson systems. For the three-dimensional Euler--Poisson equations, global small irrotational solutions were constructed by exploiting the dispersive Klein--Gordon structure generated by the Poisson coupling in \cite{Guo98}. Related global small-data theories for ion dynamics and non-neutral electron flows were developed in \cite{GP11,GMP13}. Removing pressure introduces a further degeneracy. The critical-threshold theory for pressureless Euler--Poisson equations was initiated in \cite{ELT01}. For one-dimensional pressureless Euler--Poisson equations with non-vanishing background states, a corresponding theory was developed in \cite{CKKT26}. In particular, the role of a negative homogeneous Sobolev condition associated with charge neutrality was identified. For pressureless Euler--Poisson equations with quadratic confinement, a sharp criterion for radial global smooth solutions in several dimensions was obtained in \cite{CS23}. These results show that, once either pressure or direct relaxation is lost, global regularity must rely on more delicate dispersive, geometric,
or structural mechanisms.

%
%
%
%
%
%

\subsection{Main result}
 
The novelty of the present problem lies in the simultaneous absence of pressure and damping in the same fluid. In the partially damped model studied in \cite{ZZ19}, only one velocity is directly damped, but pressure is retained in both fluids. In contrast, the undamped fluid in \eqref{main1} is also pressureless, so that both direct relaxation and acoustic coupling are absent from the same component. Compared with the standard bipolar models considered in \cite{AJ03,LY12,WW14}, where both velocities are directly dissipated, the first fluid in our system can be stabilized only through its coupling to the second fluid via the Poisson field. This additional degeneracy changes both the high-frequency dynamics and the nonlinear regularity structure. The resulting analytical difficulties and the mechanisms used to overcome them are described in the next subsection.

To state our main result, we set $c=\sqrt\gamma$ and introduce
\begin{equation}\label{newvariable}
a_1 :=\rho_1-1,\quad a_2 := \frac{2}{\gamma-1} \left(\sqrt{\gamma\rho_2^{\gamma-1}}-\sqrt\gamma\right).
\end{equation}
Then
\[
\rho_2 =\left(1+\frac{\gamma-1}{2c}a_2\right)^{\frac{2}{\gamma-1}}, \quad \mathcal R(a_2):=\rho_2-1-\frac1c a_2.
\]
We further set
\[
b:=a_1-\frac1c a_2, \quad Q:=\rho_1-\rho_2=b-\mathcal R(a_2).
\]
Since $\mathcal R(0)=\mathcal R'(0)=0$, one has $\mathcal R(a_2)=O(a_2^2)$ near the equilibrium. Thus $b$ is the linear part of the physical charge $Q$. We denote by $(a_{10},a_{20})$ the corresponding initial perturbations.

\begin{thm}\label{thm-main}
Let $N\ge 10$ and assume that $\operatorname{curl}u_{10}=0$. Define
\begin{align}\label{main-smallness}
\varepsilon_N:=&\ \|a_{20}\|_{\dot H^{-1}} +\left\|a_{10}-\frac1c a_{20}\right\|_{\dot H^{-1}} +\|a_{10}\|_{H^{N-1}} +\|(u_{10},a_{20},u_{20})\|_{H^N}.
\end{align}
Assume also that, for some $\kappa_0>0$,
\[
1+a_{10}(x)\ge\kappa_0,\quad 1+\frac{\gamma-1}{2c}a_{20}(x)\ge\kappa_0, \quad x\in\mathbb R^3.
\]
There exists $\varepsilon_0=\varepsilon_0(\gamma,\kappa_0)>0$ such that, if $\varepsilon_N\le\varepsilon_0$, then the Cauchy problem \eqref{main1} has a unique global classical solution satisfying
\begin{align*}
a_1&\in C([0,\infty);H^{N-1}) \cap C^1([0,\infty);H^{N-2}),\\
(u_1,a_2,u_2)&\in C([0,\infty);H^N) \cap C^1([0,\infty);H^{N-1}),\\
b&\in C([0,\infty);\dot H^{-1}),\quad Q \in C([0,\infty);\dot H^{-1}),\quad \nabla\Phi \in C([0,\infty);H^N).
\end{align*}
Moreover, the densities remain uniformly positive and, for all $t\ge0$,
\begin{align}\label{main-decay-statement}
\begin{aligned}
\|(a_1,a_2)(t)\|_{L^2} &\le C\varepsilon_N(1+t)^{-\frac12},\quad \|\nabla(a_1,a_2)(t)\|_{H^1} \le C\varepsilon_N(1+t)^{-1},\\
\|(\operatorname{div}u_1,\operatorname{div}u_2)(t)\|_{H^2}
&\le C\varepsilon_N(1+t)^{-1},\quad \|(a_1,u_1,a_2,u_2)(t)\|_{W^{1,\infty}} \le C\varepsilon_N(1+t)^{-\frac54},
\end{aligned}
\end{align}
where the positive constant $C$ is independent of time.
\end{thm}

\begin{rem}
The isothermal case $\gamma=1$ can be treated by replacing the density variable in \eqref{newvariable} with
\[
a_1=\rho_1-1,\quad a_2=\log\rho_2.
\]
In this case $c=1$, and the quadratic terms proportional to $(\gamma-1)/2$ in the transformed system disappear. The linear spectral structure and the nonlinear energy argument otherwise retain the same basic form. We omit the details.
\end{rem}

%
%
%
%
%
%
\subsection{Analytical challenges and proof strategy}

We now explain the main difficulties caused by the asymmetric structure of \eqref{main1} and the ideas used to overcome them. A central point is that the absence of pressure and direct relaxation in the first fluid cannot be treated merely as the loss of a single dissipative term in the energy estimate. It changes both the linear dissipative structure and the regularity hierarchy required for the nonlinear problem.

%
%
%
%
%
%
\subsubsection*{Indirect and degenerate dissipation of the first fluid.}

The first difficulty is the complete absence of direct damping in the equation for $u_1$. The physical energy identity controls the kinetic energy of $u_1$, but its dissipation acts only on $u_2$. Thus the standard energy method does not provide an estimate of the form
\[
\int_0^\infty \|u_1(t)\|_{L^2}^2\,dt<\infty,
\]
and any decay of $u_1$ must be generated indirectly through its Poisson interaction with the second fluid.

This indirect mechanism depends strongly on the frequency. Introducing the longitudinal variables
\[
q_i=\operatorname{div}u_i,
\]
we analyze the spectrum of the corresponding linearized system. At low frequencies, one mode is diffusive, while the remaining modes are uniformly stable. The spectrum is also uniformly separated from the imaginary axis at middle frequencies. At high frequencies, however, two eigenvalues approach the imaginary axis, with real parts of order $-|\xi|^{-4}$, as proved in Section \ref{S4}. Hence the damping acting on the second carrier is transferred to the first carrier through the Poisson field, but this transfer degenerates as $|\xi|\to\infty$.

The corresponding high-frequency modes do not satisfy a uniform exponential decay estimate. This regularity-loss structure is one of the principal analytical features of the system and explains why the usual dissipative estimates for bipolar Euler--Poisson systems with damping in both momentum equations cannot be applied directly. For general background on regularity-loss dissipation in hyperbolic systems with non-symmetric relaxation, see \cite{UDK12}.

%
%
%
%
%
%

\subsubsection*{Trading high regularity for time-integrable decay.}

The high-frequency degeneration creates a second difficulty: temporal decay can be recovered only at the expense of additional spatial derivatives. More precisely, for some $ \eta >0$, every $\ell\ge0$, and $|\xi|\ge1$,
\[
e^{- \eta  t|\xi|^{-4}} \lesssim (1+t)^{-\ell/4}|\xi|^\ell.
\]
Thus every additional gain in time decay requires additional high-frequency regularity. This forces us to work at a relatively high Sobolev level and to combine the Green-function analysis with the nonlinear energy method, rather than deriving decay directly from the physical energy inequality.

In the present argument, five additional derivatives are used at high frequencies to obtain the time-integrable rate
\[
(1+t)^{-5/4}.
\]
After estimating the Duhamel terms separately in the low-, middle-, and high-frequency regions, we obtain
\[
\|(a_1,u_1,a_2,u_2)(t)\|_{W^{1,\infty}} \lesssim \varepsilon_N(1+t)^{-5/4}.
\]
Consequently,
\[
\int_0^\infty \|(a_1,u_1,a_2,u_2)(t)\|_{W^{1,\infty}}\,dt<\infty.
\]
This integrable Lipschitz bound supplies the coefficient estimate needed in the high-order nonlinear energy inequality. Thus, the linear regularity-loss estimate and the nonlinear energy argument complement each other: the former provides the time-integrable lower-order norm, while the latter supplies the additional derivatives required by the high-frequency semigroup.

%
%
%
%
%
%

\subsubsection*{The charge variable and the asymmetric regularity hierarchy.}

A further difficulty comes from the interaction between the Poisson field and the two density equations. Propagating all unknowns at the same Sobolev level is not compatible with the derivative counting in the first continuity equation and the available control of the Poisson force. The appropriate variable is
\[
b=a_1-\frac1c a_2.
\]
Recall that
\[
Q=\rho_1-\rho_2=b-\mathcal R(a_2), \quad \mathcal R(a_2)=O(a_2^2).
\]
Thus $b$ is the linear part of the physical charge and retains the essential structure of $Q$. More importantly, it allows us to use the asymmetric regularity hierarchy
\[
b\in H^{N-1}, \quad (u_1,a_2,u_2)\in H^N.
\]

Although this hierarchy is asymmetric, it is adapted to the structure of the equations. At the top derivative level, the Poisson terms in the equations for $b,u_1$, and $u_2$ exhibit an exact cancellation. At the same time, the potentially derivative-losing nonlinear terms from the second fluid, in particular those involving
\[
a_2\operatorname{div}u_2 \quad\text{and}\quad a_2\nabla a_2,
\]
cancel after the corresponding differentiated equations are combined and integrated by parts. These cancellations make it possible to close the high-order energy estimate with one fewer derivative on the first density, without imposing additional regularity on $a_1$.

%
%
%
%
%
%

\subsubsection*{The undamped transverse mode and the irrotationality condition.}

The asymmetric dissipation also produces a structural obstruction in the transverse part of the first velocity. Since the Poisson force is a gradient, it acts only on the longitudinal component of $u_1$. The linearized equation consequently contains no mechanism that damps $\operatorname{curl}u_1$. This is why the assumption
\[
\operatorname{curl}u_{10}=0
\]
in Theorem \ref{thm-main} is structural rather than merely technical.

The vorticity equation shows that this condition is preserved by the nonlinear flow. Thus, we can reconstruct $u_1$ from its divergence through the Hodge decomposition and transfer the Green-function estimates for $q_1=\operatorname{div}u_1$ to the full velocity $u_1$. By contrast, the vorticity of $u_2$ is directly damped and decays exponentially. This distinction between the longitudinal and transverse dynamics is another consequence of the broken dissipative symmetry between the two carriers.

These ingredients form the global argument. The physical energy controls the physical charge and the directly damped component $u_2$, while the spectral analysis transfers dissipation to the longitudinal part of $u_1$. The regularity-loss Green-function estimates provide an integrable $W^{1,\infty}$ bound, and the $b$-based energy functional exploits the exact top-order cancellations. The resulting time-weighted bootstrap closes for sufficiently small initial data, and the continuation criterion yields the global classical solution stated in Theorem \ref{thm-main}.

%
%
%
%
%
%

We close the introduction by describing the organization of the paper. Section \ref{S2} collects the notation and preliminary estimates. Section \ref{S3} contains the reformulation, local theory, and the basic energy estimate. The spectral and Green-function analysis is carried out in Section \ref{S4}. Section \ref{S5} is devoted to the high-order energy method and the nonlinear decay estimates. The proof of Theorem \ref{thm-main} is completed in Section \ref{S6}. The detailed proof of the local well-posedness theorem is given in Appendix \ref{appendix-local}.

%
%
%
%
%
%

\section{Preliminaries}\label{S2}

%
%
%
%
%
%

\subsection{Notation}

Our notation follows that of \cite{HTWZ24}, with the frequency decomposition adapted to the low--middle--high splitting used below.

The spaces $L^p(\mathbb{R}^3)$ and $W^{k,p}(\mathbb{R}^3)$ denote the usual Lebesgue and Sobolev spaces on $\mathbb{R}^3$, with norms $\|\cdot\|_{L^p}$ and $\|\cdot\|_{W^{k,p}}$, respectively.  When $p=2$, we write $W^{k,2}(\mathbb{R}^3)=H^k(\mathbb{R}^3)$ and set
\[
\|u\|_{H^k(\mathbb{R}^3)}=\|u\|_{H^k},
\quad
\|u\|_{L^p(\mathbb{R}^3)}=\|u\|_{L^p}.
\]
We denote by $C$ a generic positive constant which may vary from line to line. The notation $f_1\lesssim f_2$ means that there exists a constant $C>0$ such that $f_1\leq C f_2$, while $f_1\sim f_2$ means that there exist positive constants $C_1,C_2$ such that
\[
f_1\leq C_1f_2,
\quad
f_2\leq C_2f_1.
\]
For an integer $k\geq0$, the symbol $\nabla^k$ denotes the collection of all
derivatives
\[
D^\ell
=\partial_{x_1}^{\ell_1}\partial_{x_2}^{\ell_2}\partial_{x_3}^{\ell_3},
\quad
|\ell|=\ell_1+\ell_2+\ell_3=k.
\]
For a function $f$, $\|f\|_X$ denotes its norm in $X$, and
$\|(f,g)\|_X:=\|f\|_X+\|g\|_X$.

The Fourier transform of $f$ is denoted by $\widehat f$ or
$\mathscr F[f]$ and is normalized by
\[
 \widehat f(\xi)
 =\mathscr F[f](\xi)
 =\int_{\mathbb R^3}f(x)e^{-ix\cdot\xi}\,dx,
 \qquad
 \mathscr F^{-1}[g](x)
 =\frac1{(2\pi)^3}\int_{\mathbb R^3}g(\xi)e^{ix\cdot\xi}\,d\xi.
\]
For $k\in\mathbb{R}$, let $\Lambda^k$ denote the Fourier multiplier
\[
\Lambda^k f =\mathscr F^{-1}\!\left(|\xi|^k\widehat f(\xi)\right).
\]
We write
$\Lambda=(-\Delta)^{1/2}$ and use the convention
\[
\|f\|_{\dot H^{-1}}=\|\Lambda^{-1}f\|_{L^2},
\quad
\|\nabla(-\Delta)^{-1}f\|_{L^2}=\|f\|_{\dot H^{-1}}.
\]
The operator $(-\Delta)^{-1}$ is understood as the Fourier multiplier $|\xi|^{-2}$.

Fix a smooth low--middle--high frequency decomposition. Choose $0<r_0<R_0$ so that $4r_0<R_0$, $R_0/2>1$, the low-frequency expansions of Lemma \ref{lem-low} are valid on $|\xi|\le 2r_0$, and the high-frequency expansions of Lemma \ref{lem-high} are valid on $|\xi|\ge R_0/2$.  This is possible after decreasing $r_0$ and increasing $R_0$ once and for all.

Take $\widehat\chi_\ell,\widehat\chi_h\in
C^\infty(\mathbb{R}^3)$, $0\leq\widehat\chi_\ell,\widehat\chi_h\leq1$, such that
\[
\widehat\chi_\ell(\xi)=1\quad (|\xi|\leq r_0),
\quad
\widehat\chi_\ell(\xi)=0\quad (|\xi|\geq2r_0),
\]
and
\[
\widehat\chi_h(\xi)=0\quad (|\xi|\leq R_0/2),
\quad
\widehat\chi_h(\xi)=1\quad (|\xi|\geq R_0).
\]
Set
\[
\widehat\chi_m=1-\widehat\chi_\ell-\widehat\chi_h.
\]
For any tempered distribution for which the following expressions are defined, set
\[ 
f^\ell=\mathscr F^{-1}\!\left(\widehat\chi_\ell\,\widehat f\right),
\quad
f^m=\mathscr F^{-1}\!\left(\widehat\chi_m\,\widehat f\right),
\quad
f^h=\mathscr F^{-1}\!\left(\widehat\chi_h\,\widehat f\right).
\] 
Thus
\[
f=f^\ell+f^m+f^h.
\]
For vector-valued functions, the decomposition is understood componentwise. This convention is compatible with the low-, middle-, and high-frequency regions used in the spectral analysis below.

%
%
%
%
%
%

\subsection{Auxiliary lemmas}

We recall some elementary inequalities used repeatedly below.

\begin{lem} \label{lem-prelim-commutator}
Let $m\geq1$ be an integer.  For the differential commutator
\[ 
[\nabla^m,f]g=\nabla^m(fg)-f\nabla^m g,
\] 
one has
\begin{equation}\label{prelim-commutator}
\big\|[\nabla^m,f]g\big\|_{L^p}
\lesssim
\|\nabla f\|_{L^{p_1}}\|\nabla^{m-1}g\|_{L^{p_2}}
+
\|\nabla^m f\|_{L^{p_3}}\|g\|_{L^{p_4}},
\end{equation}
where $p,p_2,p_3\in(1,\infty)$ and
\[ 
\frac{1}{p}
=\frac{1}{p_1}+\frac{1}{p_2}
=\frac{1}{p_3}+\frac{1}{p_4}.
\] 
Moreover, with $\Lambda=(-\Delta)^{1/2}$ and
\[
[\Lambda^m,f]g:=\Lambda^m(fg)-f\Lambda^m g,
\]
the following Kato--Ponce commutator estimate holds:
\[
\big\|[\Lambda^m,f]g\big\|_{L^2}
\lesssim
\|\nabla f\|_{L^\infty}\|\Lambda^{m-1}g\|_{L^2}
+
\|\Lambda^m f\|_{L^2}\|g\|_{L^\infty}.
\]
The estimates are understood componentwise for vector-valued functions.
\end{lem}
\begin{proof}
The differential commutator estimate \eqref{prelim-commutator} is standard; see
\cite[Lemma A.3]{Wan12}.  The estimate for $[\Lambda^m,f]g$ is the
Kato--Ponce commutator inequality; see \cite{KP88}.
\end{proof}

We next establish the structural property of $u_1$ used in the spectral analysis.

\begin{lem}\label{lem-irrotational}
	Let $(\rho_1,u_1,\Phi)$ be a sufficiently smooth solution to
	\begin{equation}\label{eq:rho1-u1-system}
		\left\{ \begin{aligned}
			&\partial_t\rho_1+\operatorname{div}(\rho_1u_1)=0,\\
			&\partial_t(\rho_1u_1)
			+\operatorname{div}(\rho_1u_1\otimes u_1)
			=\rho_1\nabla\Phi.
		\end{aligned}\right.
	\end{equation}	
	on $[0,T]\times\mathbb{R}^3$. Assume that
	\begin{equation}
		\rho_1(t,x)>0,
		\quad
		(t,x)\in[0,T]\times\mathbb{R}^3.
		\label{eq:rho1-positive}
	\end{equation}
	If the initial velocity $u_{10}$ is irrotational, namely,
	\begin{align}
	\operatorname{curl}u_{10}=0,
		\label{eq:initial-irrotational}
	\end{align}
	then 
\[ 
		\operatorname{curl}u_1(t,x)=0,
		\quad
		(t,x)\in[0,T]\times\mathbb{R}^3.
\] 
\end{lem}

\begin{proof}
	We first rewrite the momentum equation in its non-conservative
	form. A direct calculation gives
\[ 
		\partial_t(\rho_1u_1)
		+\operatorname{div}(\rho_1u_1\otimes u_1)
		=\rho_1\bigl(\partial_tu_1+(u_1\cdot\nabla)u_1\bigr)
		+u_1\bigl(\partial_t\rho_1
		+\operatorname{div}(\rho_1u_1)\bigr).
\]
	Using the continuity equation in
	\eqref{eq:rho1-u1-system}, we deduce that
\[ 
		\rho_1\bigl(\partial_tu_1+(u_1\cdot\nabla)u_1\bigr)
		=\rho_1\nabla\Phi.
\] 
	In view of the positivity assumption \eqref{eq:rho1-positive},
	it follows that
	\begin{equation}
		\partial_tu_1+(u_1\cdot\nabla)u_1=\nabla\Phi.
		\label{eq:velocity-equation}
	\end{equation}
	
	Set  $\omega_1=\operatorname{curl}u_1$.
	Taking the curl of \eqref{eq:velocity-equation}, we obtain
\[ 
		\partial_t\omega_1
		+\operatorname{curl}\bigl((u_1\cdot\nabla)u_1\bigr)=0.
\] 
	Recall the vector identity
\[ 
		\operatorname{curl}\bigl((u\cdot\nabla)u\bigr)
		=(u\cdot\nabla)\omega
		+(\operatorname{div}u)\omega
		-(\omega\cdot\nabla)u,
		\quad
		\omega=\operatorname{curl}u.
\] 
	Thus, $\omega_1$ satisfies
	\begin{equation}
		\partial_t\omega_1
		+(u_1\cdot\nabla)\omega_1
		+(\operatorname{div}u_1)\omega_1
		-(\omega_1\cdot\nabla)u_1=0.
		\label{eq:vorticity-equation}
	\end{equation}

	Let $X(t,x)$ be the flow map generated by $u_1$, namely,
\[
		\frac{d}{dt}X(t,x)=u_1\bigl(t,X(t,x)\bigr),
		\quad X(0,x)=x.
\]
	Evaluating \eqref{eq:vorticity-equation} along the trajectory gives
	\begin{align}\label{eq:vorticity-characteristic-1}
		\frac{d}{dt}\omega_1\bigl(t,X(t,x)\bigr)
		&=\bigl((\omega_1\cdot\nabla)u_1\bigr)\bigl(t,X(t,x)\bigr)
		-\bigl(\operatorname{div}u_1\bigr)\bigl(t,X(t,x)\bigr)
		\omega_1\bigl(t,X(t,x)\bigr)\notag\\
		&=\Bigl[\nabla u_1-(\operatorname{div}u_1)I\Bigr]
		\bigl(t,X(t,x)\bigr)\,
		\omega_1\bigl(t,X(t,x)\bigr).
	\end{align}
	For each fixed $x$, this is a homogeneous linear ordinary differential equation for the vector
	$\omega_1(t,X(t,x))$.

	By the initial assumption \eqref{eq:initial-irrotational},
\[
		\omega_1(0,x)
		=\operatorname{curl}u_{10}(x)
		=0.
\]
	The uniqueness of solutions to the linear ordinary differential equation \eqref{eq:vorticity-characteristic-1} yields
\[
		\omega_1\bigl(t,X(t,x)\bigr)=0,
		\quad
		t\in[0,T].
\]
	Since $X(t,\cdot)$ is a diffeomorphism as long as the solution
	remains sufficiently smooth, we conclude that
	$$
	\omega_1(t,x)=0,
	\quad
	(t,x)\in [0,T]\times\mathbb{R}^3.
	$$
	Hence
	$$
	\operatorname{curl} u_1(t,x)=0,
	\quad
	(t,x)\in[0,T]\times\mathbb{R}^3.
	$$
\end{proof}

%
%
%
%
%
%

\section{Reformulation and basic energy estimate}\label{S3}
We reformulate the system in the perturbation variables introduced in \eqref{newvariable}, state the local well-posedness result, and derive the basic physical energy estimate.

Recall that $c=\sqrt{\gamma}$. Solving the definition of $a_2$ in
\eqref{newvariable} for $\rho_2$ gives
\begin{align}\label{rho2-a2}
\rho_2
=\left(1+\frac{\gamma-1}{2c}a_2\right)^{\frac{2}{\gamma-1}}.
\end{align}
Set
\[
\mathcal R(a_2)
:=\rho_2-1-\frac1c a_2.
\]
Then $\mathcal R(0)=\mathcal R'(0)=0$, and
\begin{align}\label{physical-charge}
\rho_1-\rho_2
=\left(a_1-\frac1c a_2\right)-\mathcal R(a_2).
\end{align}
In terms of $(a_1,u_1,a_2,u_2)$, system \eqref{main1} is equivalent, as long as $1+a_1 > 0$ and 
$1+\frac{\gamma-1}{2c}a_2>0$, to

\begin{equation}\label{main2}
\left\{
\begin{aligned}
&\partial_t a_1+\operatorname{div}u_1=f_1,\\
&\partial_tu_1+\nabla(-\Delta)^{-1}\left(a_1-\frac1c a_2\right)=f_2,\\
&\partial_t a_2+c\operatorname{div}u_2=f_3,\\
&\partial_t u_2+c\nabla a_2-\nabla(-\Delta)^{-1}\left(a_1-\frac1c a_2\right)+u_2=f_4,
\end{aligned}
\right.
\end{equation}
where the nonlinear terms $f_i(i=1,2,3,4)$ are given by
\[ 
\left\{
\begin{aligned}
f_1&=-\operatorname{div}(a_1u_1),\\
f_2&=\nabla(-\Delta)^{-1}\mathcal R(a_2)-u_1\cdot\nabla u_1,\\
f_3&=-u_2\cdot\nabla a_2-\frac{\gamma-1}{2}a_2\operatorname{div}u_2,\\
f_4&=-\nabla(-\Delta)^{-1}\mathcal R(a_2)-u_2\cdot\nabla u_2
-\frac{\gamma-1}{2}a_2\nabla a_2.
\end{aligned}
\right.
\] 
We consider system \eqref{main2} with the initial data
\[ 
	(a_1,u_1,a_2,u_2)|_{t=0}=(a_{10},u_{10},a_{20},u_{20})(x),
\] 
and the far-field condition
\begin{align}\label{fd}
	\lim_{|x|\rightarrow +\infty}(a_1,u_1,a_2,u_2)(x,t)=(0,0,0,0).
\end{align}

%
%
%
%
%
%

\subsection{Local well-posedness} 
The following theorem provides the local solution and continuation criterion used in the a priori analysis. Its proof is postponed to Appendix \ref{appendix-local}.

\begin{thm} \label{thm-local}
Let {$N\geq 10$}, and set
$b_0=a_{10}-c^{-1}a_{20}$.  Assume
\[
a_{10}\in H^{N-1},\quad
(u_{10},a_{20},u_{20})\in H^N,\quad
b_0\in\dot H^{-1},
\]
and suppose that, for some $\kappa_0>0$,
\begin{align}\label{initial-positive}
1+a_{10}(x)\ge\kappa_0,
\quad
1+\frac{\gamma-1}{2c}a_{20}(x)\ge\kappa_0
\quad (x\in\mathbb R^3).
\end{align}
Then there exists $T_0>0$ and a unique solution of \eqref{main2}--\eqref{fd} on
$[0,T_0]$ satisfying
\begin{align}\label{local-class}
\begin{aligned}
a_1&\in C([0,T_0];H^{N-1})\cap C^1([0,T_0];H^{N-2}),\\
(u_1,a_2,u_2)&\in C([0,T_0];H^N)\cap C^1([0,T_0];H^{N-1}),\\
b&\in C([0,T_0];\dot H^{-1}),\quad \nabla\Phi \in C([0,T_0];H^N).
\end{aligned}
\end{align}

Moreover, the two quantities in \eqref{initial-positive} remain strictly positive on
$[0,T_0]$.
\end{thm}
\begin{rem} 
	The local solution extends globally once the continuation norm below remains bounded and the density stays positive.  Let $T^*>0$ be the maximal existence time and define
	\begin{align}\label{continuation-norm}
		\mathfrak X_N(t):=
		\|a_1(t)\|_{H^{N-1}}
		+\|(u_1,a_2,u_2)(t)\|_{H^N}
		+\left\|a_1(t)-\frac1c a_2(t)\right\|_{\dot H^{-1}}.
	\end{align}
	If $T^*<\infty$ and
	\begin{equation}\label{continuation-criterion}
		 \sup_{0\le t<T^*}\mathfrak X_N(t)<\infty,\quad \inf_{0\le t<T^*,\,x\in\mathbb R^3}(1+a_1(t,x))>0,\quad \inf_{0\le t<T^*,\,x\in\mathbb R^3}
		\left(1+\frac{\gamma-1}{2c}a_2(t,x)\right)>0,
	\end{equation}
	then the solution extends beyond $T^*$.
\end{rem}

\begin{rem} 
The Sobolev regularity in Theorem \ref{thm-local} is exactly the regularity asserted for the global solution in Theorem \ref{thm-main}.  The additional assumption $a_{20}\in\dot H^{-1}$ appearing in \eqref{main-smallness} is deliberately not part of the local well-posedness hypothesis: it is a low-frequency assumption used only in the Green-function decay estimates, in particular in Proposition \ref{prop-linear-decay}. The negative norm needed for local solvability of the Poisson field is instead the charge condition $b_0\in\dot H^{-1}$ (equivalently, $Q_0\in\dot H^{-1}$ on bounded positive Sobolev sets).
\end{rem}

Under the assumptions of Theorem \ref{thm-main}, let $T^*>0$ be the maximal existence time of the local solution. All estimates below are derived on an arbitrary interval
\[
[0,T]\subset[0,T^*)
\]
with constants independent of $T$.

Fix $\delta>0$ sufficiently small and assume on $[0,T]$ that
\begin{align}\label{small-neighborhood}
\sup_{0\le t\le T}\Big(
\|a_1(t)\|_{H^2}+\|a_2(t)\|_{H^3}
+\|u_1(t)\|_{H^3}+\|u_2(t)\|_{H^3}
+\|b(t)\|_{\dot H^{-1}}\Big)\le\delta.
\end{align}
Sobolev embedding and \eqref{small-neighborhood} then give
\[
1+a_1(t,x)\ge\frac12,
\quad
1+\frac{\gamma-1}{2c}a_2(t,x)\ge\frac12.
\]

%
%
%
%
%
%

\subsection{Basic $L^2$ energy estimate}
The following estimate controls the physical charge and identifies the direct dissipation supplied by the second velocity.
 \begin{prop}\label{p1}
Assume that the hypotheses of Theorem \ref{thm-main} hold. Let $(a_1,u_1,a_2,u_2)$ be a classical solution of \eqref{main2}--\eqref{fd} on $[0,T]$ satisfying  \eqref{small-neighborhood}. Then
 \begin{align*}
 	&\|(u_1,a_2,u_2)(t)\|_{L^2}^2+\|\nabla(-\Delta)^{-1}(a_1-\frac{1}{c}a_2)(t)\|_{L^2}^2
 	+\int_0^t\|u_2(\tau)\|_{L^2}^2d\tau\\
 	&\quad \leq C(\|(a_{20},u_{10},u_{20})\|_{L^2}^2+\|a_{10}-\frac{1}{c}a_{20}\|_{\dot{H}^{-1}}^2),
 \end{align*}	
 where $C$ is a positive constant independent of time.
\end{prop}	

\begin{proof}
 Define the internal energy
\[
H(\rho_2):=\rho_2\int_1^{\rho_2}\frac{P(y)-P(1)}{y^2}\,dy.
\]
Since $P'(1)=\gamma>0$, one has $H(1)=H'(1)=0$ and $H''(1)>0$.  Hence, as long as $\rho_2$ stays in a sufficiently small neighborhood of $1$,
\begin{align}\label{H-equivalence}
C^{-1}(\rho_2-1)^2\le H(\rho_2)\le C(\rho_2-1)^2.
\end{align}
Multiplying the two momentum equations in \eqref{main1} by $u_1$ and $u_2$, respectively, using the two continuity equations, and integrating over $\mathbb R^3$, we obtain
\[
 \frac{d}{dt}\int_{\mathbb R^3}
\left(\frac12\rho_1|u_1|^2+\frac12\rho_2|u_2|^2+H(\rho_2)\right)dx -\int_{\mathbb R^3}(\rho_1u_1-\rho_2u_2)\cdot\nabla\Phi\,dx
+\int_{\mathbb R^3}\rho_2|u_2|^2\,dx=0.
\]
 Indeed, by the continuity equations and
$\Delta\Phi=\rho_1-\rho_2$,
\begin{align*}
-\int_{\mathbb R^3}(\rho_1u_1-\rho_2u_2)\cdot\nabla\Phi\,dx
&=\int_{\mathbb R^3}\big[\operatorname{div}(\rho_1u_1)-\operatorname{div}(\rho_2u_2)\big]\Phi\,dx\\
&=-\int_{\mathbb R^3}\partial_t(\rho_1-\rho_2)\Phi\,dx\\
&=\frac12\frac{d}{dt}
\|\nabla(-\Delta)^{-1}(\rho_1-\rho_2)\|_{L^2}^2.
\end{align*}
Hence
\begin{equation}\label{physical-basic-energy}
 \frac{d}{dt}\int_{\mathbb R^3}\left(
\frac12\rho_1|u_1|^2+\frac12\rho_2|u_2|^2+H(\rho_2)
+\frac12|\nabla(-\Delta)^{-1}(\rho_1-\rho_2)|^2\right)dx +\int_{\mathbb R^3}\rho_2|u_2|^2\,dx=0.
\end{equation}

In the perturbation variables, \eqref{physical-charge} gives
\[
Q=\rho_1-\rho_2=b-\mathcal R(a_2).
\]
Thus the physical charge $Q$ is not exactly $b$.  Since $\mathcal R(a_2)=O(a_2^2)$, the difference is quadratic.  Using the Sobolev inequality
$L^{6/5}(\mathbb R^3)\hookrightarrow\dot H^{-1}(\mathbb R^3)$ and interpolation inequality between $L^2$ and $L^6$, and the Sobolev
inequality, we obtain
\[ 
\|\mathcal R(a_2)\|_{\dot H^{-1}}
\le C\|\mathcal R(a_2)\|_{L^{\frac{6}{5}}} \le C\|a_2\|_{L^{\frac{12}{5}}}^2 \le C\|a_2\|_{L^2}^{\frac{3}{2}}\|a_2\|_{L^6}^{\frac{1}{2}}
\le C\|a_2\|_{L^2}^{\frac{3}{2}}\|\nabla a_2\|_{L^2}^{\frac{1}{2}}
\le C\delta\|a_2\|_{L^2},
\] 
where the last inequality follows from \eqref{small-neighborhood}.  Therefore
\begin{align}\label{charge-b-equivalence}
\|b\|_{\dot H^{-1}}
 \le \|\rho_1-\rho_2\|_{\dot H^{-1}}+C\delta\|a_2\|_{L^2},\quad \|\rho_1-\rho_2\|_{\dot H^{-1}} \le \|b\|_{\dot H^{-1}}+C\delta\|a_2\|_{L^2}.
\end{align}
The same bound holds at $t=0$.  From \eqref{rho2-a2} and the smallness assumption,
\[
\|\rho_2-1\|_{L^2}\sim\|a_2\|_{L^2},
\]
and the positivity bounds imply that the weighted kinetic and damping terms in
\eqref{physical-basic-energy} are equivalent to the corresponding unweighted $L^2$ norms.

Integrating \eqref{physical-basic-energy} in time, using \eqref{H-equivalence}--\eqref{charge-b-equivalence}, and taking $\delta$ sufficiently small, we arrive at
\begin{align*}
 \|(u_1,a_2,u_2)(t)\|_{L^2}^2 +\|b(t)\|_{\dot H^{-1}}^2 +\int_0^t\|u_2(\tau)\|_{L^2}^2\,d\tau \le C\left( \|(u_{10},a_{20}, u_{20})\|_{L^2}^2 +\|b_0\|_{\dot H^{-1}}^2\right),
\end{align*}
and $\|b\|_{\dot H^{-1}}=\|\nabla(-\Delta)^{-1}b\|_{L^2}$. This completes the proof.
\end{proof}

%
%
%
%
%
%

\section{Linear spectral analysis}\label{S4}
Proposition \ref{p1} supplies direct dissipation only for $u_2$. To recover decay for the first velocity, we analyze the longitudinal part of the linearized system. The linearization of \eqref{main2} is
\begin{equation}\label{lin-system}
\left\{
\begin{aligned}
&\partial_t a_1+\operatorname{div}u_1=0,\\
&\partial_t u_1+\nabla(-\Delta)^{-1}\left(a_1-\frac{1}{c}a_2\right)=0,\\
&\partial_t a_2+c\operatorname{div}u_2=0,\\
&\partial_t u_2+c\nabla a_2-\nabla(-\Delta)^{-1}\left(a_1-\frac{1}{c}a_2\right)+u_2=0.
\end{aligned}
\right.
\end{equation}
By Lemma \ref{lem-irrotational}, the first velocity remains
irrotational. Thus it is enough to introduce the longitudinal
variables
\begin{align*}
q_1=\operatorname{div}u_1,\quad q_2=\operatorname{div}u_2.
\end{align*}
Then the linearized system \eqref{lin-system} becomes
\begin{equation}\label{lin-aq}
\left\{
\begin{aligned}
&\partial_t a_1+q_1=0,\\
&\partial_t q_1-a_1+\frac{1}{c}a_2=0,\\
&\partial_t a_2+cq_2=0,\\
&\partial_t q_2+c\Delta a_2+a_1-\frac{1}{c}a_2+q_2=0.
\end{aligned}
\right.
\end{equation}
Set
\begin{align*}
U(x,t)=(a_1,q_1,a_2,q_2)^T(x,t),\quad
U(0,x)=U_0(x)=(a_{10},q_{10},a_{20},q_{20})^T(x).
\end{align*}
Then
\begin{equation}\label{lin-matrix}
\partial_t U+\mathcal{L}U=0,
\end{equation}
where
\[
\mathcal{L}=\begin{pmatrix}
0&1&0&0\\
-1&0&\frac{1}{c}&0\\
0&0&0&c\\
1&0&c\Delta-\frac{1}{c}&1
\end{pmatrix}.
\]
Taking the Fourier transform gives
\[
\widehat{\mathcal{L}}(\xi)=\begin{pmatrix}
0&1&0&0\\
-1&0&\frac{1}{c}&0\\
0&0&0&c\\
1&0&-c|\xi|^2-\frac{1}{c}&1
\end{pmatrix}.
\]
Hence the characteristic polynomial of $-\widehat{\mathcal{L}}(\xi)$ is
\begin{equation}\label{char-eq}
p(\lambda,\xi) =\det\left(\lambda I+\widehat{\mathcal{L}}(\xi)\right) =\lambda^4+\lambda^3+\left(2+c^2|\xi|^2\right)\lambda^2+\lambda+c^2|\xi|^2.
\end{equation}

We analyze the roots of \eqref{char-eq} in the low-, middle-, and high-frequency regions.



\subsection{Low-frequency analysis}
We begin with the spectrum near $\xi=0$, where one eigenvalue generates the diffusive mode.

\begin{lem}\label{lem-low}
There exists $r_0>0$ sufficiently small such that, for $|\xi|\le 2r_0$, the four roots of \eqref{char-eq} satisfy
\begin{align*}
\lambda_1&=e_0+e_2|\xi|^2+O(|\xi|^4),\quad \lambda_2=j_0+j_2|\xi|^2+O(|\xi|^4),\\
\lambda_3&=z_0+z_2|\xi|^2+O(|\xi|^4),\quad \lambda_4=-c^2|\xi|^2+O(|\xi|^4),
\end{align*}
where $e_0,j_0,z_0$ are the three roots of
\[
\lambda^3+\lambda^2+2\lambda+1=0
\]
and
\begin{align*}
e_2=\frac{-(e_0^2+1)c^2}{4e_0^3+3e_0^2+4e_0+1},\quad
j_2=\frac{-(j_0^2+1)c^2}{4j_0^3+3j_0^2+4j_0+1},\quad
z_2=\frac{-(z_0^2+1)c^2}{4z_0^3+3z_0^2+4z_0+1}.
\end{align*}
Moreover, there exists $\eta_1>0$ such that
\[
\operatorname{Re}\lambda_j(\xi)\le-\eta_1\quad (j=1,2,3),
\qquad
\operatorname{Re}\lambda_4(\xi)\le-\eta_1|\xi|^2,
\quad |\xi|\le2r_0.
\]
\end{lem}

\begin{proof}
Put $s=|\xi|^2$ and define
\begin{align*}
F(s,\lambda)=\lambda^4+\lambda^3+(2+c^2s)\lambda^2+\lambda+c^2s.
\end{align*}
At $s=0$,
\[
F(0,\lambda)=\lambda\left(\lambda^3+\lambda^2+2\lambda+1\right).
\]
The root $\lambda=0$ is simple, and the three nonzero roots $e_0,j_0,z_0$ are also simple. For any $\mu\in\{e_0,j_0,z_0\}$, the Implicit Function Theorem gives
\[
\lambda(s)=\mu+\lambda'(0)s+O(s^2).
\]
Since
\[
F_s(s,\lambda)=c^2(\lambda^2+1),\quad
F_\lambda(s,\lambda)=4\lambda^3+3\lambda^2+2(2+c^2s)\lambda+1,
\]
we have
\[
\lambda'(0)=-\frac{F_s(0,\mu)}{F_\lambda(0,\mu)}
=\frac{-(\mu^2+1)c^2}{4\mu^3+3\mu^2+4\mu+1}.
\]
This yields the expansions of $\lambda_1,\lambda_2,\lambda_3$.  For the root near zero, writing $\lambda_4=d_1s+O(s^2)$ and substituting into $F(s,\lambda_4)=0$ yield $d_1=-c^2$. Hence,
\[
\lambda_4=-c^2|\xi|^2+O(|\xi|^4).
\]
The cubic $\lambda^3+\lambda^2+2\lambda+1$ is Hurwitz, since its coefficients are positive and $1\cdot 2>1$. Thus the three nonzero roots have negative real parts, i.e.,
\[
\max\{\operatorname{Re}e_0,\operatorname{Re}j_0,\operatorname{Re}z_0\}<0.
\]
Since all four roots of $F(0,\lambda)$ are simple, the Implicit Function Theorem implies that the corresponding eigenvalue branches are analytic in $s=|\xi|^2$ in a neighborhood of $s=0$. Hence, by choosing $r_0>0$ sufficiently small so that $0\le s\le 4r_0^2$ lies entirely in this common neighborhood, the above Taylor expansions, together with their remainder estimates, hold uniformly for $|\xi|\le 2r_0$. The strict negativity of the three nonzero roots and the expansion $\lambda_4=-c^2|\xi|^2+O(|\xi|^4)$ then give the asserted bounds with a common $\eta_1>0$ due to the smallness of $r_0$. Thus there exists $\eta_1>0$ such that
	\[
	\operatorname{Re}\lambda_j(\xi)\le-\eta_1\quad (j=1,2,3),
	\qquad
	\operatorname{Re}\lambda_4(\xi)\le-\eta_1|\xi|^2,
	\quad |\xi|\le2r_0.
	\]
This completes the proof.
\end{proof}

\subsection{Middle-frequency analysis} We then investigate the uniform spectral stability of the eigenvalues in the middle-frequency part.

\begin{lem}\label{lem-middle-gap}
	Let $0<r_0<R_0<\infty$ be fixed. There exists a constant
	$ \eta_2 = \eta_2 (c,r_0,R_0)>0$ such that the four roots of
	\eqref{char-eq} satisfy
	\begin{equation}\label{middle-spectral-gap}
		\max_{1\le j\le4}\operatorname{Re}\lambda_j(\xi)
		\le - 2\eta_2, 
		\qquad r_0\le |\xi|\le R_0.
	\end{equation}
\end{lem}

\begin{proof}
	Set $s=c^2|\xi|^2>0$. The characteristic polynomial \eqref{char-eq} becomes
	\[
	p_s(\lambda)
	=\lambda^4+\lambda^3+(2+s)\lambda^2+\lambda+s.
	\]
	Its Hurwitz determinants are
	\[
	\Delta_1=1,\qquad
	\Delta_2=(2+s)-1=1+s,\qquad
	\Delta_3=(2+s)-s-1=1,\qquad
	\Delta_4=s\Delta_3=s.
	\]
	Since all four determinants are positive, the Routh--Hurwitz
	criterion implies that every root of $p_s$ has strictly negative
	real part for each $s>0$.
	
	Suppose that \eqref{middle-spectral-gap} fails. Then there exist
	$s_n\in[c^2r_0^2,c^2R_0^2]$ and roots $\lambda_n$ of $p_{s_n}$
	such that $\operatorname{Re}\lambda_n\to0$. The  root bound
	gives $|\lambda_n|\le 3+c^2R_0^2$. Passing to a subsequence, we have
	\[
	s_n\to s_*\in[c^2r_0^2,c^2R_0^2],
	\qquad \lambda_n\to\lambda_*.
	\]
	By continuity,
	\[
	p_{s_*}(\lambda_*)=0,
	\qquad \operatorname{Re}\lambda_*=0.
	\]
	Since $s_*>0$, this contradicts the strict negativity established
	above. Thus the negative real parts are bounded away from zero uniformly on the annulus. Taking $\eta_2$ to be half of this positive lower bound gives \eqref{middle-spectral-gap}. This completes the proof. 
\end{proof}


\subsection{High-frequency analysis}
We next consider the limit $|\xi|\to\infty$, where two modes have degenerate real parts and produce the regularity-loss behavior.

\begin{lem}\label{lem-high}
There exists $R_0>0$ sufficiently large such that, for $|\xi|\ge R_0/2$, the four roots satisfy
\begin{align*}
\lambda_1&=i-\frac{i}{2c^2}|\xi|^{-2}
-\left(\frac12+\frac{i}{8}\right)c^{-4}|\xi|^{-4}+O(|\xi|^{-6}),\\
\lambda_2&=-i+\frac{i}{2c^2}|\xi|^{-2}
-\left(\frac12-\frac{i}{8}\right)c^{-4}|\xi|^{-4}+O(|\xi|^{-6}),\\
\lambda_3&=-\frac12+ic|\xi|+\frac{3i}{8c}|\xi|^{-1}+O(|\xi|^{-2}),\quad \lambda_4 =-\frac12-ic|\xi|-\frac{3i}{8c}|\xi|^{-1}+O(|\xi|^{-2}).
\end{align*}
Hence there exists $\eta_3 >0$ such that
\begin{align}\label{high-eigen-decay}
\begin{aligned}
|e^{\lambda_1t}|+|e^{\lambda_2t}| &\le C\exp\left(-\eta_3  t|\xi|^{-4}\right),\\
|e^{\lambda_3t}|+|e^{\lambda_4t}| &\le Ce^{-\eta_3  t},
\quad |\xi|\ge R_0/2.
\end{aligned}
\end{align}
\end{lem}

\proof
Let $\theta=|\xi|^{-1}$. Multiplying \eqref{char-eq} by $\theta^2$ gives
\[
F_2(\theta,\lambda) =\theta^2\lambda^4+\theta^2\lambda^3+(2\theta^2+c^2)\lambda^2+\theta^2\lambda+c^2=0.
\]
At $\theta=0$, the two bounded roots are $\lambda=\pm i$. Since
$\partial_\lambda F_2(0,\pm i)\ne0$, the Implicit Function Theorem gives expansions in even powers of $\theta$. Substitution yields
\begin{align*}
\lambda_1 =i-\frac{i}{2c^2}\theta^2 -\left(\frac12+\frac{i}{8}\right)c^{-4}\theta^4+O(\theta^6),\quad \lambda_2=-i+\frac{i}{2c^2}\theta^2 -\left(\frac12-\frac{i}{8}\right)c^{-4}\theta^4+O(\theta^6).
\end{align*}
For the remaining two roots, put $\bar\lambda=\theta\lambda$. Then
\[
\bar\lambda^4+\theta\bar\lambda^3+(2\theta^2+c^2)\bar\lambda^2
+\theta^3\bar\lambda+c^2\theta^2=0.
\]
At $\theta=0$, the two nonzero roots are $\bar\lambda=\pm ic$. Expanding around these roots gives
\begin{align*}
\bar\lambda_3 =ic-\frac12\theta+\frac{3i}{8c}\theta^2+O(\theta^3),\quad \bar\lambda_4 =-ic-\frac12\theta-\frac{3i}{8c}\theta^2+O(\theta^3).
\end{align*}
Dividing by $\theta$ gives the stated expansions of $\lambda_3$ and $\lambda_4$, i.e.,
\begin{align*}
	\lambda_3 =-\frac12+ic\theta^{-1}+\frac{3i}{8c}\theta+O(\theta^2),\quad \lambda_4 =-\frac12-ic\theta^{-1}-\frac{3i}{8c}\theta+O(\theta^2).
\end{align*}

Since the above expansions are valid for $\theta=|\xi|^{-1}$ in a
neighborhood of $\theta=0$, we choose $R_0>0$ sufficiently large so that
$0<\theta\le 2/R_0$ is contained in this neighborhood whenever
$|\xi|\ge R_0/2$. In particular, the remainder estimates are uniform
in this region. For $\lambda_1$ and $\lambda_2$, we have
\[
\operatorname{Re}\lambda_{1,2}
=-\frac{1}{2c^4}\theta^4+O(\theta^6).
\]
Since $O(\theta^6)=o(\theta^4)$ as $\theta\to0$, $R_0$ can be chosen
so that there exists $\eta_{31}>0$ such that
\[
\operatorname{Re}\lambda_{1,2}
\le -\eta_{31}\theta^4
=-\eta_{31}|\xi|^{-4},
\quad |\xi|\ge R_0/2.
\]
Similarly,
\[
\operatorname{Re}\lambda_{3,4}
=-\frac12+O(\theta^2).
\]
Since $O(\theta^2)=o(1)$ as $\theta\to0$, the same $R_0$ can be chosen
so that, for some $\eta_{32}>0$,
\[
\operatorname{Re}\lambda_{3,4}\le -\eta_{32},
\quad |\xi|\ge R_0/2.
\]
Taking $\eta_3 =\min\{\eta_{31},\eta_{32}\}>0$, we consequently obtain
\[
|e^{\lambda_1t}|+|e^{\lambda_2t}|
\le C\exp\left(-\eta_3  t|\xi|^{-4}\right),
\quad
|e^{\lambda_3t}|+|e^{\lambda_4t}|
\le Ce^{-\eta_3  t},
\]
which proves \eqref{high-eigen-decay}. This completes the proof.
\endproof

\begin{rem} 
Estimate \eqref{high-eigen-decay} shows that the two bounded high-frequency modes do not admit a uniform spectral gap. Their damping degenerates as $|\xi|\to\infty$, whereas the remaining two modes are uniformly exponentially stable.

Consequently, high-frequency time decay must be obtained by exchanging spatial regularity for temporal decay, as quantified in \eqref{regularity-loss-ineq}. This derivative--time tradeoff is the reason for working at the higher Sobolev level $N\ge10$ in the nonlinear estimates.
\end{rem}

\begin{rem}  
We compare the high-frequency structure above with that of the system in which pressure is also present in the first fluid. Suppose that a pressure term is added to the first momentum equation in \eqref{main1}, normalized so that its linearization contributes $\nabla a_1$. Introducing
\[
q_1=\operatorname{div}u_1,
\quad
q_2=\operatorname{div}u_2,
\]
we obtain the Fourier symbol
\[
\widehat{\mathcal L}_p(\xi)=
\begin{pmatrix}
0&1&0&0\\
-(1+|\xi|^2)&0&\frac{1}{c}&0\\
0&0&0&c\\
1&0&-c|\xi|^2-\frac{1}{c}&1
\end{pmatrix}.
\]
Writing $s=|\xi|^2$, the corresponding characteristic polynomial is
\[
p_p(\lambda,s)
=
\lambda^4+\lambda^3
+\bigl(2+(1+c^2)s\bigr)\lambda^2
+(1+s)\lambda
+c^2s^2+(1+c^2)s.
\]

At low frequencies, the spectrum has the same qualitative structure as in Lemma \ref{lem-low}: three roots remain uniformly separated from the imaginary axis, while the remaining root satisfies
\[
\lambda(\xi)
=-(1+c^2)|\xi|^2+O(|\xi|^4).
\]
For fixed $c>1$, a direct expansion at high frequencies gives
\[
\begin{aligned}
\lambda_{1,2}(\xi)
={}&
\pm i|\xi|
\pm\frac{i}{2|\xi|}
\mp\frac{i(c^2+3)}
{8(c^2-1)|\xi|^3}
-\frac{1}
{2(c^2-1)^2|\xi|^4}
+O(|\xi|^{-5}),\\
\lambda_{3,4}(\xi)
={}&
-\frac12
\pm ic|\xi|
\pm\frac{3i}{8c|\xi|}
\pm\frac{i(55c^2+9)}
{128c^3(c^2-1)|\xi|^3}
+\frac{1}
{2(c^2-1)^2|\xi|^4}
+O(|\xi|^{-5}).
\end{aligned}
\]
Consequently,
\[
\operatorname{Re}\lambda_{1,2}(\xi)
\sim
-\frac{1}{2(c^2-1)^2}|\xi|^{-4},
\quad
\operatorname{Re}\lambda_{3,4}(\xi)
\longrightarrow-\frac12
\]
as $|\xi|\to\infty$.

The additional pressure changes the oscillatory behavior of the weakly
damped modes: their imaginary parts, which remain bounded in the
pressureless system, become of order $|\xi|$. It also increases the
magnitude of the leading damping coefficient. Indeed, since $c>1$,
\[
\frac{1}{2(c^2-1)^2}
>
\frac{1}{2c^4},
\]
where $1/(2c^4)$ is the corresponding coefficient in
Lemma \ref{lem-high}. The order of the weak dissipation, however, remains
unchanged. In both systems, the high-frequency semigroup contains a factor
of the form
\[
e^{-\eta_3  t|\xi|^{-4}}.
\]
At the linearized level, the additional pressure improves the
high-frequency damping quantitatively but does not eliminate the
regularity-loss mechanism.
\end{rem}

%
%
%
%
%
%

\subsection{Green-function estimates}
We now combine the eigenvalue expansions with the corresponding spectral projectors to obtain pointwise estimates for the Green matrix.

Let
\[
\widehat G(t,\xi)=e^{-t\widehat{\mathcal L}(\xi)}.
\]
Choose $r_0$ small and $R_0$ large so that all four roots are simple for $|\xi|\le2r_0$ and $|\xi|\ge R_0/2$.  Then
\begin{equation}\label{green}
\widehat G(t,\xi)=\sum_{i=1}^4e^{\lambda_it}P_i(\xi),
\quad
P_i(\xi)=\frac{\operatorname{adj}(\lambda_i I+\widehat{\mathcal L}(\xi))}{\Delta_i},
\end{equation}
where
\[
\Delta_i=\partial_\lambda p(\lambda_i,\xi)=\prod_{j\ne i}(\lambda_i-\lambda_j).
\]
A direct computation gives
\begin{equation}\label{projection-exact}
P_i=\frac1{\Delta_i}
\begin{pmatrix}
\lambda_i(c^2|\xi|^2+\lambda_i^2+\lambda_i+1)
&-(c^2|\xi|^2+\lambda_i^2+\lambda_i+1)
&\frac{\lambda_i+1}{c}&-1\\
 c^2|\xi|^2+\lambda_i^2+\lambda_i
&\lambda_i(c^2|\xi|^2+\lambda_i^2+\lambda_i+1)
&-\frac{\lambda_i(\lambda_i+1)}{c}&\lambda_i\\
 c\lambda_i&-c&(\lambda_i+1)(\lambda_i^2+1)&-c(\lambda_i^2+1)\\
-\lambda_i^2&\lambda_i
&\frac{(1+c^2|\xi|^2)\lambda_i^2+c^2|\xi|^2}{c}
&\lambda_i(\lambda_i^2+1)
\end{pmatrix}.
\end{equation}
Inserting the eigenvalue expansions into \eqref{projection-exact} yields the following orders.

For $|\xi|\le 2r_0$, by Lemma \ref{lem-low}, one has
\begin{align}\label{low-projection}
P_1=O(1),\quad P_2=O(1),\quad P_3=O(1),
\end{align}
and
\begin{align}\label{low-P4}
P_4=
\begin{pmatrix}
O(|\xi|^2)&O(1)&O(1)&O(1)\\
O(|\xi|^4)&O(|\xi|^2)&O(|\xi|^2)&O(|\xi|^2)\\
O(|\xi|^2)&O(1)&O(1)&O(1)\\
O(|\xi|^4)&O(|\xi|^2)&O(|\xi|^2)&O(|\xi|^2)
\end{pmatrix}.
\end{align}
For the charge variable $a_1-c^{-1}a_2$, the stronger cancellation
\begin{align}\label{low-charge-cancel}
\left(1,0,-\frac1c,0\right)P_4
=\left(O(|\xi|^6),O(|\xi|^4),O(|\xi|^4),O(|\xi|^4)\right).
\end{align}
Indeed, the first entry in \eqref{low-charge-cancel} equals
\[
\frac{\lambda_4(c^2|\xi|^2+\lambda_4^2+\lambda_4)}{\Delta_4}=O(|\xi|^6),
\]
while the remaining three entries are $O(|\xi|^4)$.

For $|\xi|\ge R_0/2$, by Lemma \ref{lem-high}, the two weakly damped projectors satisfy
\begin{align}\label{high-P12}
P_1,\ P_2=
\begin{pmatrix}
O(1)&O(1)&O(|\xi|^{-2})&O(|\xi|^{-2})\\
O(1)&O(1)&O(|\xi|^{-2})&O(|\xi|^{-2})\\
O(|\xi|^{-2})&O(|\xi|^{-2})&O(|\xi|^{-4})&O(|\xi|^{-4})\\
O(|\xi|^{-2})&O(|\xi|^{-2})&O(|\xi|^{-4})&O(|\xi|^{-4})
\end{pmatrix},
\end{align}
whereas the strongly damped projectors satisfy
\begin{align}\label{high-P34}
P_3,\ P_4=
\begin{pmatrix}
O(|\xi|^{-4})&O(|\xi|^{-5})&O(|\xi|^{-2})&O(|\xi|^{-3})\\
O(|\xi|^{-3})&O(|\xi|^{-4})&O(|\xi|^{-1})&O(|\xi|^{-2})\\
O(|\xi|^{-2})&O(|\xi|^{-3})&O(1)&O(|\xi|^{-1})\\
O(|\xi|^{-1})&O(|\xi|^{-2})&O(|\xi|)&O(1)
\end{pmatrix}.
\end{align}
For convenience, set
\[
\mathcal A=(\alpha_{ij})=
\begin{pmatrix}
2&0&0&0\\
4&2&2&2\\
2&0&0&0\\
4&2&2&2
\end{pmatrix},\quad
\mathcal S=(\sigma_{ij})=
\begin{pmatrix}
0&0&2&2\\
0&0&2&2\\
2&2&4&4\\
2&2&4&4
\end{pmatrix},
\]
\[
\mathcal B=(\beta_{ij})=
\begin{pmatrix}
-4&-5&-2&-3\\
-3&-4&-1&-2\\
-2&-3&0&-1\\
-1&-2&1&0
\end{pmatrix}.
\]
Also set
\[
\widehat G^\nu(t,\xi)=\widehat\chi_\nu(\xi)\widehat G(t,\xi),
\quad \nu\in\{\ell,m,h\}.
\]

\begin{lem}[Green-function estimates]\label{lem-green-estimates}
 With the positive constants $\eta_1$, $\eta_2$, and $\eta_3$ from Lemmas \ref{lem-low}, \ref{lem-middle-gap}, and \ref{lem-high}, there exists $C>0$ such that, for all $t\ge0$, the following estimates hold.

\smallskip
\noindent\emph{(i) Low-frequency estimate.} For $|\xi|\le2r_0$,
\begin{align}\label{low-G-pointwise}
|\widehat G_{ij}(t,\xi)|
\le Ce^{-\eta_1  t}+C|\xi|^{\alpha_{ij}}e^{-\eta_1 |\xi|^2t},
\quad 1\le i,j\le4,
\end{align}
and
\[
\left|\left(1,0,-\frac1c,0\right)\widehat G(t,\xi)\right| \le Ce^{-\eta_1  t}(1,1,1,1) +Ce^{-\eta_1 |\xi|^2t} \big(|\xi|^6,|\xi|^4,|\xi|^4,|\xi|^4\big).
\]
If
\[
\widehat{K_\mu^\ell(t)f}(\xi)
=\widehat\chi_\ell(\xi)|\xi|^\mu e^{-\eta_1 |\xi|^2t}\widehat f(\xi),
\quad \mu\ge0,
\]
then, for every integer $k\ge0$,
\begin{align}
\|\nabla^kK_\mu^\ell(t)f\|_{L^2}
&\le C(1+t)^{-\frac34-\frac{k+\mu}{2}}\|f\|_{L^1}, \notag \\
\|\nabla^kK_\mu^\ell(t)f\|_{L^2}
&\le C(1+t)^{-\frac{k+\mu}{2}}\|f\|_{L^2},  \notag \\
\label{low-kernel-L1inf}
\|\nabla^kK_\mu^\ell(t)f\|_{L^\infty}
&\le C(1+t)^{-\frac32-\frac{k+\mu}{2}}\|f\|_{L^1},\\
\label{low-kernel-L2inf}
\|\nabla^kK_\mu^\ell(t)f\|_{L^\infty}
&\le C(1+t)^{-\frac34-\frac{k+\mu}{2}}\|f\|_{L^2}.
\end{align}

\smallskip
\noindent\emph{(ii) Middle-frequency estimate.} For $r_0\le|\xi|\le R_0$,
\begin{align}\label{middle-decay}
|\widehat G_{ij}(t,\xi)|\le Ce^{-\eta_2   t},
\quad 1\le i,j\le4.
\end{align}
Consequently, for every integer $k\ge0$,
\begin{align}\label{middle-G-operator}
\|\nabla^k(G^m(t)*f)\|_{L^2}
+\|\nabla^k(G^m(t)*f)\|_{L^\infty}
\le C_k e^{-\eta_2   t}\|f\|_{L^2}.
\end{align}

\smallskip
\noindent\emph{(iii) High-frequency part.} For $|\xi|\ge R_0/2$,
\begin{align}\label{high-G-pointwise}
|\widehat G_{ij}(t,\xi)|
\le C|\xi|^{-\sigma_{ij}}e^{-\eta_3  t|\xi|^{-4}}
+C|\xi|^{\beta_{ij}}e^{-\eta_3  t},
\quad 1\le i,j\le4.
\end{align}
Furthermore, for every $\ell\ge0$,
\begin{align}\label{regularity-loss-ineq}
e^{-\eta_3  t|\xi|^{-4}}
\le C_\ell(1+t)^{-\frac\ell4}|\xi|^\ell,
\end{align}
and hence, for every integer $k\ge0$ and $\ell\ge1$,
\begin{align}\label{high-linear-general}
\|\nabla^k(G^h(t)*f)\|_{L^2}
\le C_\ell(1+t)^{-\frac\ell4}\|f\|_{H^{k+\ell}}.
\end{align}
\end{lem}

\begin{proof}
The low-frequency bounds follow from \eqref{green}, \eqref{low-projection}--\eqref{low-P4}, and \eqref{low-charge-cancel}; the four kernel estimates are the standard heat-kernel multiplier bounds.

 For middle frequencies, Lemma \ref{lem-middle-gap} gives
$\operatorname{Re}\lambda_j(\xi)\le-2\eta_2$ on the fixed annulus.
The matrices $-\widehat{\mathcal L}(\xi)$ are uniformly bounded there.
A unitary Schur decomposition and successive integration of the resulting
upper triangular system therefore yield
\[
\|\widehat G(t,\xi)\|
\le C(1+t+t^2+t^3)e^{-2\eta_2t}
\le Ce^{-\eta_2t}.
\]
This proves \eqref{middle-decay}; \eqref{middle-G-operator} follows from
Plancherel and the bounded frequency support.

Finally, \eqref{high-G-pointwise} follows from \eqref{green}, \eqref{high-P12}--\eqref{high-P34}, and Lemma \ref{lem-high}. Since
\[
(1+t)^{\frac{\ell}{4}}|\xi|^{-\ell}e^{-\eta_3  t|\xi|^{-4}}
\le C_\ell,
\]
we obtain \eqref{regularity-loss-ineq}, and \eqref{high-linear-general} follows by Plancherel. The factor $|\xi|$ in the $(4,3)$ entry of \eqref{high-P34} is absorbed by $\ell\ge1$.

This completes the proof.
\end{proof}

%
%
%
%
%
%

\subsection{Linear decay estimates}

We next convert the Green-function bounds into decay estimates for the linearized solution. The divergence structure of the initial velocities provides an additional low-frequency factor.

By definition,
\[
q_{10}=\operatorname{div}u_{10},\quad q_{20}=\operatorname{div}u_{20}.
\]
For every $i=1,2,3,4$, integration by parts in the convolution gives
\begin{align}\label{IBP-Green}
G_{i2}(t)*q_{10}
=\sum_{m=1}^3\partial_mG_{i2}(t)*u_{10,m},\quad
G_{i4}(t)*q_{20}
=\sum_{m=1}^3\partial_mG_{i4}(t)*u_{20,m}.
\end{align}
Thus one derivative is transferred to the Green function, providing one additional low-frequency factor $|\xi|$.

\begin{prop}\label{prop-linear-decay}
Let $N\ge10$ and assume
\[
\mathcal I_N :=\|a_{20}\|_{\dot H^{-1}} +\|a_{10}\|_{H^{N-1}} +\|(u_{10},a_{20},u_{20})\|_{H^N}<\infty.
\]
Then $\mathcal I_N\le \varepsilon_N$, with $\varepsilon_N$ defined in \eqref{main-smallness}.
Let $U=(a_1,q_1,a_2,q_2)^T$ solve \eqref{lin-aq}. Then, for all $t\ge0$,
\begin{align}\label{linear-density-decay}
\|(a_1,a_2)(t)\|_{L^2}
&\le C(1+t)^{-\frac12}\mathcal I_N,\\
\|\nabla(a_1,a_2)(t)\|_{H^1}
&\le C(1+t)^{-1}\mathcal I_N, \notag\\
\label{linear-q-decay}
\|q_1(t)\|_{H^2}+\|q_2(t)\|_{H^2}
&\le C(1+t)^{-\frac54}\mathcal I_N.
\end{align}
In particular,
\begin{align}\label{linear-q-integrable}
\int_0^\infty\|q_1(t)\|_{H^2}^2\,dt
\le C\mathcal I_N^2.
\end{align}
If, in addition, $\operatorname{curl}u_{10}=0$, then
\begin{align}\label{linear-u1-integrable}
\int_0^\infty\|\nabla u_1(t)\|_{H^2}^2\,dt
\le C\mathcal I_N^2.
\end{align}
\end{prop}

\proof
We split the solution into low-, middle-, and high-frequency parts.

For the low-frequency part, \eqref{low-G-pointwise} and \eqref{IBP-Green} imply, for $m=0,1,2$,
\begin{align}\label{low-density-final}
\|\nabla^m(a_1^\ell,a_2^\ell)(t)\|_{L^2}
\le C(1+t)^{-\frac{m+1}{2}}
\left(\|a_{10}\|_{L^2}+\|a_{20}\|_{\dot H^{-1}}+\|(u_{10},u_{20})\|_{L^2}\right).
\end{align}
Indeed, the slow projector in the first and third rows has order $O(|\xi|^2)$ on $a_{10}$, order $O(1)$ on $a_{20}$, and order $O(1)$ on $q_{10},q_{20}$. The assumption $a_{20}\in\dot H^{-1}$ supplies one low-frequency factor $|\xi|$, while \eqref{IBP-Green} supplies one factor $|\xi|$ for $q_{10}$ and $q_{20}$. Hence the slow part of each density contains at least one power of $|\xi|$.

Similarly, the second and fourth rows of \eqref{low-P4} contain at least $|\xi|^2$ on the last three columns and $|\xi|^4$ on the first column. Therefore, for $m=0,1,2$,
\begin{align}\label{low-q-final}
\|\nabla^m(q_1^\ell,q_2^\ell)(t)\|_{L^2}
\le C(1+t)^{-\frac{m+3}{2}}
\left(\|a_{10}\|_{L^2}+\|a_{20}\|_{\dot H^{-1}}+\|(u_{10},u_{20})\|_{L^2}\right).
\end{align}
The exponentially damped low-frequency modes are absorbed in \eqref{low-density-final}--\eqref{low-q-final}.

For the middle-frequency part, \eqref{middle-decay} gives exponential decay. For the high-frequency part, we choose $\ell=5$ in \eqref{high-linear-general}, which yields the time-integrable decay rate $(1+t)^{-5/4}$ needed in the subsequent nonlinear estimates. Since $N\ge10$, $q_{i0}=\operatorname{div}u_{i0}$, and $0\le m\le2$,
\[
\|U_0\|_{H^{m+5}}
\le C\left(\|a_{10}\|_{H^{N-1}}+\|(u_{10},a_{20},u_{20})\|_{H^N}\right)
\le C\mathcal I_N.
\]
Therefore,
\begin{align}\label{high-final}
\|\nabla^mU^h(t)\|_{L^2}
\le C(1+t)^{-\frac54}\mathcal I_N,
\quad 0\le m\le2.
\end{align}
Combining \eqref{low-density-final}, \eqref{low-q-final}, the middle-frequency exponential estimate, and \eqref{high-final} yields \eqref{linear-density-decay}--\eqref{linear-q-decay}.

Since $5/2>1$, \eqref{linear-q-decay} implies \eqref{linear-q-integrable}.  If $\operatorname{curl}u_{10}=0$, the linearized first velocity remains irrotational, and the div--curl estimate gives
\[
\|\nabla u_1(t)\|_{H^2}\le C\|\operatorname{div}u_1(t)\|_{H^2}
=C\|q_1(t)\|_{H^2},
\]
which, together with \eqref{linear-q-integrable}, gives \eqref{linear-u1-integrable}. This completes the proof.
\endproof

%
%
%
%
%
%

\section{High-order energy and nonlinear decay estimates}\label{S5}
This section combines the nonlinear energy method with the Green-function estimates. We first formulate the Duhamel representation, then derive the high-order energy inequality and estimate the nonlinear source terms in the three frequency regions.

Set
\[
q_1=\operatorname{div}u_1,\quad q_2=\operatorname{div}u_2.
\]
Taking divergence in the two velocity equations of \eqref{main2}, we obtain
\begin{equation}\label{nonlinear-aq}
\left\{
\begin{aligned}
&\partial_t a_1+q_1=g_1,\\
&\partial_t q_1-a_1+\frac1c a_2=g_2,\\
&\partial_t a_2+cq_2=g_3,\\
&\partial_t q_2+c\Delta a_2+a_1-\frac1c a_2+q_2=g_4,
\end{aligned}
\right.
\end{equation}
where the nonlinear terms $g_i(i=1,2,3,4)$ are given by
\begin{align}\label{nonlinear-g}
\begin{aligned}
g_1&=-\operatorname{div}(a_1u_1),\quad g_2=-\mathcal R(a_2)-\operatorname{div}(u_1\cdot\nabla u_1),\quad g_3=-u_2\cdot\nabla a_2-\frac{\gamma-1}{2}a_2q_2,\\ g_4&=\mathcal R(a_2)-\operatorname{div}(u_2\cdot\nabla u_2)
-\frac{\gamma-1}{2}\operatorname{div}(a_2\nabla a_2).
\end{aligned}
\end{align}
Here $\mathcal R(a_2)=O(a_2^2)$ for small $a_2$. Let
\[
U=(a_1,q_1,a_2,q_2)^T,\quad F=(g_1,g_2,g_3,g_4)^T.
\]
Then the system \eqref{nonlinear-aq} can be written in the following vector form:
\[
\partial_tU+\mathcal LU=F,
\]
where $\mathcal L$ is the operator defined in \eqref{lin-matrix}.  Hence
\begin{equation}\label{nonlinear-duhamel}
U(t)=G(t)*U_0+\int_0^tG(t-\tau)*F(\tau)\,d\tau.
\end{equation}

Due to the high-frequency regularity loss, we work at the fixed Sobolev level $N\geq 10$. Define the high-order energy
\[
\mathcal E_N(t) := \|a_1(t)\|_{H^{N-1}}^2 +\|(u_1,a_2,u_2)(t)\|_{H^N}^2 +\left\|a_1(t)-\frac1c a_2(t)\right\|_{\dot H^{-1}}^2 +\left\|a_1(t)-\frac1c a_2(t)\right\|_{H^{N-1}}^2.
\]
Assume temporarily that
\begin{equation}\label{high-bootstrap}
\sup_{0\le\tau\le t}\mathcal E_N(\tau)^{\frac{1}{2}}\le 2\delta
\end{equation}
with $\delta>0$ sufficiently small.

Define the time-weighted functional
\begin{align}\label{M-new}
\begin{aligned}
\mathcal M(t)=\sup_{0\le\tau\le t}\Big\{& (1+\tau)^{\frac{1}{2}}\|(a_1,a_2)(\tau)\|_{L^2} +(1+\tau)\|\nabla(a_1,a_2)(\tau)\|_{H^1}\\
&\quad +(1+\tau)\|(q_1,q_2)(\tau)\|_{H^2} +(1+\tau)^{\frac{5}{4}}\mathcal Z(\tau)\Big\},
\end{aligned}
\end{align}
where
\begin{align}\label{Z-def}
\mathcal Z(t)=
\|(a_1,a_2)(t)\|_{W^{1,\infty}}
+\|(u_1,u_2)(t)\|_{W^{1,\infty}}.
\end{align}

We deduce from the definition of $\mathcal{M}(t)$ that
\begin{align*}
\|(a_1,a_2)(t)\|_{L^2}&\leq (1+t)^{-\frac{1}{2}}\mathcal{M}(t),\quad \|\nabla(a_1,a_2)(t)\|_{H^1} \leq (1+t)^{-1}\mathcal{M}(t),\\
\|(q_1,q_2)(t)\|_{L^2}&\leq (1+t)^{-1}\mathcal{M}(t),\quad \mathcal{Z}(t) \leq (1+t)^{-\frac{5}{4}}\mathcal{M}(t).
\end{align*}
It remains to prove that $\mathcal M(t)$ is bounded uniformly in time.

\begin{lem}\label{lem-high-order-energy}
Let {$N\geq 10$}. Under the assumptions of Theorem \ref{thm-main} and the
smallness assumption \eqref{high-bootstrap}, the smooth solution of
\eqref{main2} satisfies
\begin{align}\label{high-order-tame}
\mathcal E_N(t)+\int_0^t\|u_2(\tau)\|_{H^N}^2\,d\tau
\le C\mathcal E_N(0)
+C\int_0^t\mathcal Z(\tau)\mathcal E_N(\tau)\,d\tau .
\end{align}
If $\int_0^t\mathcal Z(\tau)\,d\tau\le C_0$, then
\begin{align}\label{high-order-bound}
\mathcal E_N(t)\le C\mathcal E_N(0)e^{CC_0}.
\end{align}
\end{lem}

\proof
Set
\[
\alpha=\frac{\gamma-1}{2},\quad \mathscr R=\nabla\Lambda^{-1}.
\]
$\mathscr R$ is the vector of Riesz transforms, and
\begin{align*}
\nabla(-\Delta)^{-1}=\mathscr R\Lambda^{-1},\quad
\operatorname{div}v=\mathscr R\cdot\Lambda v.
\end{align*}
Recall also that
\begin{align*}
\mathcal R(a_2)=\left(1+\frac{\gamma-1}{2c}a_2\right)^{\frac{2}{\gamma-1}}
-1-\frac1c a_2,
\end{align*}
and hence
\begin{align}\label{R-composition-high}
\|\mathcal R(a_2)\|_{H^m}
\le C\|a_2\|_{L^\infty}\|a_2\|_{H^m},
\quad 0\le m\le N-1,
\end{align}
provided $\|a_2\|_{L^\infty}$ is sufficiently small.  This follows from
$\mathcal R(0)=\mathcal R'(0)=0$ and the Moser composition estimate.

We first derive an equation for $b$.  From the first and third equations in
\eqref{main2},
\begin{align*}
\partial_ta_1+\operatorname{div}u_1
&=-u_1\cdot\nabla a_1-a_1\operatorname{div}u_1,\\
\partial_ta_2+c\operatorname{div}u_2
&=-u_2\cdot\nabla a_2-\alpha a_2\operatorname{div}u_2.
\end{align*}
Using $a_1=b+c^{-1}a_2$, we obtain
\begin{equation}\label{high-b-equation}
\partial_tb+u_1\cdot\nabla b+\operatorname{div}(u_1-u_2)=H_b,
\end{equation}
where
\begin{align}\label{high-Hb}
H_b=&-\frac1c(u_1-u_2)\cdot\nabla a_2
-b\operatorname{div}u_1-\frac1c a_2\operatorname{div}u_1
+\frac{\alpha}{c}a_2\operatorname{div}u_2.
\end{align}
The remaining three equations can be written as
\begin{equation}\label{high-rewritten-system}
\left\{
\begin{aligned}
&\partial_tu_1+\mathscr R\Lambda^{-1}b+u_1\cdot\nabla u_1
=\mathscr R\Lambda^{-1}\mathcal R(a_2),\\
&\partial_ta_2+c\operatorname{div}u_2+u_2\cdot\nabla a_2
+\alpha a_2\operatorname{div}u_2=0,\\
&\partial_tu_2+c\nabla a_2-\mathscr R\Lambda^{-1}b+u_2
+u_2\cdot\nabla u_2+\alpha a_2\nabla a_2
=-\mathscr R\Lambda^{-1}\mathcal R(a_2).
\end{aligned}
\right.
\end{equation}

By Lemma \ref{lem-prelim-commutator}, together with the standard Sobolev
product inequality, for every integer $m\ge1$ we have
\begin{align}\label{high-commutator}
\|[\Lambda^m,f]g\|_{L^2}
&\le C\big(\|\nabla f\|_{L^\infty}\|\Lambda^{m-1}g\|_{L^2}
+\|\Lambda^mf\|_{L^2}\|g\|_{L^\infty}\big),\\
\label{high-product}
\|fg\|_{H^m}
&\le C\big(\|f\|_{L^\infty}\|g\|_{H^m}
+\|g\|_{L^\infty}\|f\|_{H^m}\big).
\end{align}
For $m=0$ we use the ordinary $L^2$ product estimate. Since
$N\ge10>5/2+1$, Sobolev embedding and \eqref{high-bootstrap} imply
\[
\mathcal Z(t)\le C\mathcal E_N(t)^{\frac{1}{2}}\le C\delta.
\]
In what follows, all the $L^\infty$ and $W^{1,\infty}$ norms are
controlled directly by $\mathcal Z(t)$ defined in \eqref{Z-def}.

Fix $1\le s\le N$.  Apply $\Lambda^s$ to the three equations in
\eqref{high-rewritten-system}, apply $\Lambda^{s-1}$ to \eqref{high-b-equation}, and take
the $L^2$ inner products with $\Lambda^su_1$, $\Lambda^sa_2$,
$\Lambda^su_2$, and $\Lambda^{s-1}b$, respectively.  The two linear coupling
mechanisms cancel exactly.  Indeed,
\[
c\langle \Lambda^s\operatorname{div}u_2,\Lambda^sa_2\rangle +c\langle \Lambda^s\nabla a_2,\Lambda^su_2\rangle=0,
\]
and, since each Riesz transform is skew-adjoint on $L^2$,
\[
 \langle \mathscr R\Lambda^{s-1}b,\Lambda^s(u_1-u_2)\rangle
+\langle \mathscr R\cdot\Lambda^s(u_1-u_2),\Lambda^{s-1}b\rangle=0.
\]
Therefore,
\begin{align}\label{high-level-energy-identity}
&\frac12\frac{d}{dt}\Big(
\|\Lambda^su_1\|_{L^2}^2+\|\Lambda^sa_2\|_{L^2}^2
+\|\Lambda^su_2\|_{L^2}^2+\|\Lambda^{s-1}b\|_{L^2}^2\Big)
+\|\Lambda^su_2\|_{L^2}^2\notag\\
&\quad =I_{1,s}+I_{2,s}+I_{3,s}+I_{4,s},
\end{align}
where 
\begin{align*}
I_{1,s}:={}&-\big\langle\Lambda^s(u_1\cdot\nabla u_1),\Lambda^su_1\big\rangle
-\big\langle\Lambda^s(u_2\cdot\nabla a_2),\Lambda^sa_2\big\rangle\\
&-\big\langle\Lambda^s(u_2\cdot\nabla u_2),\Lambda^su_2\big\rangle
-\big\langle\Lambda^{s-1}(u_1\cdot\nabla b),\Lambda^{s-1}b\big\rangle,\\
I_{2,s}:={}&-\alpha\big\langle\Lambda^s(a_2\operatorname{div}u_2),\Lambda^sa_2\big\rangle
-\alpha\big\langle\Lambda^s(a_2\nabla a_2),\Lambda^su_2\big\rangle,\\
I_{3,s}:={}&\big\langle\mathscr R\Lambda^{s-1}\mathcal R(a_2),
\Lambda^s(u_1-u_2)\big\rangle,\quad I_{4,s}:=\big\langle\Lambda^{s-1}H_b,\Lambda^{s-1}b\big\rangle.
\end{align*}
We estimate these four terms separately.

\medskip
\noindent\emph{Step 1: the estimate of $I_{1,s}$.}
For $j=1,2$, write
\begin{align*}
\Lambda^s(u_j\cdot\nabla u_j)
=u_j\cdot\nabla\Lambda^su_j+[\Lambda^s,u_j]\cdot\nabla u_j.
\end{align*}
After integration by parts and using \eqref{high-commutator},
\begin{align}\label{high-u-transport}
\left|\langle\Lambda^s(u_j\cdot\nabla u_j),\Lambda^su_j\rangle\right|
\le C\|\nabla u_j\|_{L^\infty}\|u_j\|_{H^s}^2.
\end{align}
For the transport term in the $a_2$ equation,
\begin{align*}
 \left|\langle\Lambda^s(u_2\cdot\nabla a_2),\Lambda^sa_2\rangle\right| 
& \le \frac12\|\operatorname{div}u_2\|_{L^\infty}
\|\Lambda^sa_2\|_{L^2}^2
+\|[\Lambda^s,u_2]\cdot\nabla a_2\|_{L^2}\|\Lambda^sa_2\|_{L^2}\\
& \le C\big(\|\nabla u_2\|_{L^\infty}\|a_2\|_{H^s}^2
+\|\nabla a_2\|_{L^\infty}\|u_2\|_{H^s}\|a_2\|_{H^s}\big).
\end{align*}
Finally, put $r=s-1$.  If $r=0$, direct integration by parts gives
\begin{align*}
\left|\langle u_1\cdot\nabla b,b\rangle\right|
\le \frac12\|\operatorname{div}u_1\|_{L^\infty}\|b\|_{L^2}^2.
\end{align*}
If $r\ge1$, then
\begin{align*}
\Lambda^r(u_1\cdot\nabla b)
=u_1\cdot\nabla\Lambda^rb+[\Lambda^r,u_1]\cdot\nabla b,
\end{align*}
and hence
\begin{align}\label{high-b-transport}
\left|\langle\Lambda^{s-1}(u_1\cdot\nabla b),\Lambda^{s-1}b\rangle\right|
\le C\|\nabla u_1\|_{L^\infty}\|b\|_{H^{s-1}}^2
+C\|\nabla b\|_{L^\infty}\|u_1\|_{H^{s-1}}\|b\|_{H^{s-1}}.
\end{align}
Since $b=a_1-c^{-1}a_2$, all coefficients in
\eqref{high-u-transport}--\eqref{high-b-transport} are bounded by $C\mathcal Z(t)$.
Hence
\begin{align}\label{high-I1}
|I_{1,s}|\le C\mathcal Z(t)\Big(
\|u_1\|_{H^s}^2+\|a_2\|_{H^s}^2+\|u_2\|_{H^s}^2
+\|b\|_{H^{s-1}}^2\Big).
\end{align}

\medskip
\noindent\emph{Step 2: the estimate of $I_{2,s}$.} Note that
\begin{align*}
I_{2,s}=- \alpha\langle\Lambda^s(a_2\operatorname{div}u_2),\Lambda^sa_2\rangle
-\alpha\langle\Lambda^s(a_2\nabla a_2),\Lambda^su_2\rangle.
\end{align*}
We decompose
\begin{align*}
\Lambda^s(a_2\operatorname{div}u_2) =a_2\Lambda^s\operatorname{div}u_2
+[\Lambda^s,a_2]\operatorname{div}u_2,\quad \Lambda^s(a_2\nabla a_2) =a_2\nabla\Lambda^sa_2+[\Lambda^s,a_2]\nabla a_2.
\end{align*}
The two potentially derivative-losing pieces cancel after one integration by parts:
\begin{align}\label{high-a2-u2-cancel}
 -\int_{\mathbb R^3}a_2\Lambda^s\operatorname{div}u_2\,\Lambda^sa_2\,dx
-\int_{\mathbb R^3}a_2\nabla\Lambda^sa_2\cdot\Lambda^su_2\,dx =\int_{\mathbb R^3}\nabla a_2\cdot\Lambda^su_2\,\Lambda^sa_2\,dx.
\end{align}
The commutator terms satisfy, by \eqref{high-commutator},
\begin{align*}
\|[\Lambda^s,a_2]\operatorname{div}u_2\|_{L^2}
&\le C\big(\|\nabla a_2\|_{L^\infty}\|u_2\|_{H^s}
+\|\nabla u_2\|_{L^\infty}\|a_2\|_{H^s}\big),\\
\|[\Lambda^s,a_2]\nabla a_2\|_{L^2}
&\le C\|\nabla a_2\|_{L^\infty}\|a_2\|_{H^s}.
\end{align*}
Combining these estimates with \eqref{high-a2-u2-cancel}, we get
\begin{align}\label{high-I2}
|I_{2,s}|\le C\mathcal Z(t)
\big(\|a_2\|_{H^s}^2+\|u_2\|_{H^s}^2\big).
\end{align}

\medskip
\noindent\emph{Step 3: the estimate of $I_{3,s}$.}
For $I_{3,s}$, since the Riesz transforms are bounded on $L^2$,
\eqref{R-composition-high} gives
\begin{align}\label{I3s}
\begin{aligned}
|I_{3,s}|
&\le C\|\mathcal R(a_2)\|_{H^{s-1}} \big(\|u_1\|_{H^s}+\|u_2\|_{H^s}\big)\\
&\le C\|a_2\|_{L^\infty}\|a_2\|_{H^{s-1}} \big(\|u_1\|_{H^s}+\|u_2\|_{H^s}\big)\\
&\le C\mathcal Z(t) \big(\|a_2\|_{H^s}^2+\|u_1\|_{H^s}^2+\|u_2\|_{H^s}^2\big).
\end{aligned}
\end{align}
In particular, the operator $\nabla(-\Delta)^{-1}$ gains one derivative, and hence
this term only requires $\mathcal R(a_2)\in H^{s-1}$.

\medskip
\noindent\emph{Step 4: the estimate of $I_{4,s}$.}
Using \eqref{high-product} and \eqref{high-Hb}, with $r=s-1$, we obtain
\[
\|H_b\|_{H^{s-1}} \le C\mathcal Z(t)\Big( \|u_1\|_{H^s}+\|u_2\|_{H^s}+\|a_2\|_{H^s}+\|b\|_{H^{s-1}} \Big).
\]
For example,
\begin{align*}
\|(u_1-u_2)\cdot\nabla a_2\|_{H^{s-1}}
&\le C\big(\|u_1-u_2\|_{L^\infty}\|a_2\|_{H^s}
+\|\nabla a_2\|_{L^\infty}\|u_1-u_2\|_{H^{s-1}}\big),\\
\|b\operatorname{div}u_1\|_{H^{s-1}}
&\le C\big(\|b\|_{L^\infty}\|u_1\|_{H^s}
+\|\nabla u_1\|_{L^\infty}\|b\|_{H^{s-1}}\big),
\end{align*}
and the terms $a_2\operatorname{div}u_j$ are treated in exactly the same way.
Hence
\begin{align}\label{high-I4}
|I_{4,s}|
&\le C\mathcal Z(t)\Big(
\|u_1\|_{H^s}^2+\|u_2\|_{H^s}^2+\|a_2\|_{H^s}^2
+\|b\|_{H^{s-1}}^2\Big).
\end{align}

Combining \eqref{high-I1}, \eqref{high-I2}, \eqref{I3s}, and
\eqref{high-I4} with \eqref{high-level-energy-identity}, we obtain, for every
$1\le s\le N$,
\begin{align}\label{high-s-est}
\begin{aligned}
&\frac{d}{dt}\Big( \|\Lambda^su_1\|_{L^2}^2+\|\Lambda^sa_2\|_{L^2}^2 +\|\Lambda^su_2\|_{L^2}^2+\|\Lambda^{s-1}b\|_{L^2}^2\Big) +\|\Lambda^su_2\|_{L^2}^2\\
&\quad\le C\mathcal Z(t)\Big( \|u_1\|_{H^s}^2+\|a_2\|_{H^s}^2+\|u_2\|_{H^s}^2 +\|b\|_{H^{s-1}}^2\Big).
\end{aligned}
\end{align}
Summing \eqref{high-s-est} over $s=1,\ldots,N$ yields
\begin{align}\label{high-summed-est}
\mathfrak E_N^{\rm h}(t)
+\int_0^t\|\nabla u_2(\tau)\|_{H^{N-1}}^2\,d\tau\le \mathfrak E_N^{\rm h}(0)
+C\int_0^t\mathcal Z(\tau)\mathcal E_N(\tau)\,d\tau,
\end{align}
where
\begin{align*}
\mathfrak E_N^{\rm h}(t)
=\sum_{s=1}^N\Big(
\|\Lambda^su_1\|_{L^2}^2+\|\Lambda^sa_2\|_{L^2}^2
+\|\Lambda^su_2\|_{L^2}^2+\|\Lambda^{s-1}b\|_{L^2}^2\Big).
\end{align*}

By \eqref{high-bootstrap}, the smallness assumption \eqref{small-neighborhood}
is satisfied for $\delta>0$ sufficiently small, and hence
Proposition \ref{p1} applies. Combining Proposition \ref{p1} with
\eqref{high-summed-est}, and using
\[
\|\nabla(-\Delta)^{-1}b\|_{L^2}=\|b\|_{\dot H^{-1}},\quad
\sum_{s=1}^N\|\Lambda^{s-1}b\|_{L^2}^2\sim\|b\|_{H^{N-1}}^2,
\]
together with $a_1=b+c^{-1}a_2$ and
\[
\|a_1\|_{H^{N-1}}
\le C\big(\|b\|_{H^{N-1}}+\|a_2\|_{H^{N-1}}\big),
\]
we obtain \eqref{high-order-tame}. Finally, Gr\"onwall's inequality gives
\[
\mathcal E_N(t)
\le C\mathcal E_N(0)
\exp\left(C\int_0^t\mathcal Z(\tau)\,d\tau\right),
\]
and hence \eqref{high-order-bound} whenever
$\int_0^t\mathcal Z(\tau)\,d\tau\le C_0$.
\endproof

We next estimate the nonlinear terms in the Duhamel formula.

\begin{lem}\label{lem-nonlinear-source}
Assume \eqref{high-bootstrap}, $\varepsilon_N\le1$, and
$\mathcal M(t)<\infty$. Then, for
$0\le\tau\le t$,
\begin{align}\label{source-L1}
\|F(\tau)\|_{L^1}
&\le C\big(\delta\mathcal M(t)+\mathcal M(t)^2\big)(1+\tau)^{-1}
+C\varepsilon_Ne^{-\frac{\tau}{2}},\\
\label{source-L2}
\|F(\tau)\|_{L^2}
&\le C\big(\delta\mathcal M(t)+\mathcal M(t)^2\big)(1+\tau)^{-\frac{3}{2}},\\
\label{source-H8}
\|F(\tau)\|_{H^8}
&\le C\delta\mathcal M(t)(1+\tau)^{-\frac{5}{4}}.
\end{align}
In addition, with
\[
F_2^{\rm div}:=u_1\cdot\nabla u_1,\quad
F_4^{\rm div}:=u_2\cdot\nabla u_2+\frac{\gamma-1}{2}a_2\nabla a_2,
\]
one has
\begin{align}\label{vector-source-L1L2}
\begin{aligned}
\|F_2^{\rm div}(\tau)\|_{L^1}+\|F_4^{\rm div}(\tau)\|_{L^1} &\le C\big(\delta\mathcal M(t)+\mathcal M(t)^2\big)(1+\tau)^{-1} +C\varepsilon_Ne^{-\frac{\tau}{2}},\\
\|F_2^{\rm div}(\tau)\|_{L^2}+\|F_4^{\rm div}(\tau)\|_{L^2} &\le C\big(\delta\mathcal M(t)+\mathcal M(t)^2\big)(1+\tau)^{-\frac{3}{2}}.
\end{aligned}
\end{align}
\end{lem}
\begin{proof} We first record the div--curl bounds for the two velocities.  By Lemma \ref{lem-irrotational},
$\operatorname{curl}u_1=0$ for all times.  Hence
\begin{align}\label{u1-from-q1}
\|\nabla u_1(\tau)\|_{H^2}
\le C\|q_1(\tau)\|_{H^2}
\le C\mathcal M(t)(1+\tau)^{-1}.
\end{align}
For the second velocity, set $\omega_2=\operatorname{curl}u_2$.  Taking curl of the fourth equation in \eqref{main2} gives
\begin{align}\label{omega2-eq}
\partial_t\omega_2+\omega_2
+(u_2\cdot\nabla)\omega_2
+(\operatorname{div}u_2)\omega_2
-(\omega_2\cdot\nabla)u_2=0,
\end{align}
since the curl of every gradient term, including $a_2\nabla a_2=\frac12\nabla(a_2^2)$, vanishes.  Applying derivatives up to order two to \eqref{omega2-eq}, using the transport commutator estimate, and taking the $L^2$ inner product with the corresponding derivatives of $\omega_2$, we obtain
\begin{align*}
\frac12\frac{d}{dt}\|\omega_2\|_{H^2}^2+\|\omega_2\|_{H^2}^2
\le C\|\nabla u_2\|_{H^2}\|\omega_2\|_{H^2}^2.
\end{align*}

The high-order bootstrap implies $C\|\nabla u_2\|_{H^2}\le C_1\delta$.
Choosing $\delta$ so that $C_1\delta\le\frac12$, we obtain
\[
\frac{d}{d\tau}\|\omega_2(\tau)\|_{H^2}^2
+\|\omega_2(\tau)\|_{H^2}^2\le0.
\]
Consequently, 
\begin{align}\label{omega2-decay}
\|\omega_2(\tau)\|_{H^2}
\le Ce^{-\frac{\tau}{2}}\|\omega_{20}\|_{H^2}
\le  Ce^{-\frac{\tau}{2}}\|u_{20}\|_{H^3}\le Ce^{-\frac{\tau}{2}}\varepsilon_N.
\end{align}
The div--curl estimate then gives
\begin{align}\label{u2-from-q2}
\|\nabla u_2(\tau)\|_{H^2}
\le C\big(\|q_2(\tau)\|_{H^2}+\|\omega_2(\tau)\|_{H^2}\big)
\le C\mathcal M(t)(1+\tau)^{-1}
+Ce^{-\frac{\tau}{2}}\varepsilon_N.
\end{align}

We also use Proposition \ref{p1} and the high-order bootstrap in the form
\begin{align}\label{uniform-L2-source}
\|(u_1,u_2,a_2)(\tau)\|_{L^2}
+\|(a_1,u_1,a_2,u_2)(\tau)\|_{H^9}
\le C\delta.
\end{align}

We next prove the $L^1$ estimates.  Since
$g_1=-a_1q_1-u_1\cdot\nabla a_1$, we have
\begin{align*}
\|g_1\|_{L^1} \le \|a_1\|_{L^2}\|q_1\|_{L^2} +\|u_1\|_{L^2}\|\nabla a_1\|_{L^2} \le C\big(\mathcal M(t)^2+\delta\mathcal M(t)\big)(1+\tau)^{-1}.
\end{align*}
For $g_2$, use
\begin{align}\label{div-convection-id}
\operatorname{div}(u\cdot\nabla u)
=u\cdot\nabla\operatorname{div}u
+\sum_{i,j=1}^3\partial_i u_j\,\partial_j u_i.
\end{align}
Together with $\|\mathcal R(a_2)\|_{L^1}\le C\|a_2\|_{L^2}^2$, this gives
\begin{align*}
\|g_2\|_{L^1} \le C\|a_2\|_{L^2}^2 +C\|u_1\|_{L^2}\|\nabla q_1\|_{L^2} +C\|\nabla u_1\|_{L^2}^2 \le C\big(\delta\mathcal M(t)+\mathcal M(t)^2\big)(1+\tau)^{-1}.
\end{align*}
Similarly,
\begin{align*}
\|g_3\|_{L^1} \le \|u_2\|_{L^2}\|\nabla a_2\|_{L^2} +C\|a_2\|_{L^2}\|q_2\|_{L^2} \le C\big(\delta\mathcal M(t)+\mathcal M(t)^2\big)(1+\tau)^{-1}.
\end{align*}
Finally, using \eqref{div-convection-id} and
$\operatorname{div}(a_2\nabla a_2)=|\nabla a_2|^2+a_2\Delta a_2$, we obtain
\begin{align*}
\|g_4\|_{L^1} &\le C\|a_2\|_{L^2}^2 +C\|u_2\|_{L^2}\|\nabla q_2\|_{L^2} +C\|\nabla u_2\|_{L^2}^2 +C\|\nabla a_2\|_{L^2}^2 +C\|a_2\|_{L^2}\|\Delta a_2\|_{L^2}\\
&\le C\big(\delta\mathcal M(t)+\mathcal M(t)^2\big)(1+\tau)^{-1}
+C\varepsilon_Ne^{-\frac{\tau}{2}}.
\end{align*}
Indeed, \eqref{u2-from-q2} and $\varepsilon_N\le1$ give
\[
\|\nabla u_2\|_{L^2}^2\le C\mathcal M(t)^2(1+\tau)^{-2}+C\varepsilon_N^2e^{-\tau}\le   C\mathcal M(t)^2(1+\tau)^{-2}+C\varepsilon_Ne^{-\tau}.
\]
 Summing these estimates gives \eqref{source-L1}.

For the $L^2$ bound, we use the $W^{1,\infty}$ component of $\mathcal M(t)$.  The corresponding estimates are
\begin{align*}
\|g_1\|_{L^2}
&\le C\big(\|a_1\|_{L^\infty}\|q_1\|_{L^2}
+\|u_1\|_{L^\infty}\|\nabla a_1\|_{L^2}\big),\\
\|g_2\|_{L^2}
&\le C\|a_2\|_{L^\infty}\|a_2\|_{L^2}
+C\|u_1\|_{L^\infty}\|\nabla q_1\|_{L^2}
+C\|\nabla u_1\|_{L^\infty}\|\nabla u_1\|_{L^2}.
\end{align*}
The corresponding estimates for $g_3$ and $g_4$ follow from the same identities, with \eqref{u2-from-q2} and $\min\{\varepsilon_N,\delta\}\le\delta$ used for the transverse part of $u_2$. Since $\mathcal Z(\tau)\le \mathcal M(t)(1+\tau)^{-\frac{5}{4}}$, while the $L^2$ density decays like $(1+\tau)^{-1/2}$ and the differentiated quantities in \eqref{M-new} decay like $(1+\tau)^{-1}$, each product is bounded by $C(\delta\mathcal M(t)+\mathcal M(t)^2)(1+\tau)^{-3/2}$. The contribution generated by the transverse part of $u_2$ is bounded directly by
\[
C\delta\mathcal M(t)e^{-\frac{\tau}{2}}(1+\tau)^{-\frac{5}{4}}
\le C\delta\mathcal M(t)(1+\tau)^{-\frac{3}{2}},
\]
since the exponential factor dominates every algebraic power.  Hence  we obtain  \eqref{source-L2}.

The vector sources in \eqref{vector-source-L1L2} are estimated similarly.  Indeed,
\begin{align*}
\|u_j\cdot\nabla u_j\|_{L^1} \le \|u_j\|_{L^2}\|\nabla u_j\|_{L^2},\quad \|a_2\nabla a_2\|_{L^1} \le \|a_2\|_{L^2}\|\nabla a_2\|_{L^2},
\end{align*}
and the right-hand side is bounded by
\[
C(\delta\mathcal M(t)+\mathcal M(t)^2)(1+\tau)^{-1}
+C\varepsilon_Ne^{-\frac{\tau}{2}}
\]
using \eqref{u1-from-q1}, \eqref{u2-from-q2}, the uniform $L^2$ bound
\eqref{uniform-L2-source}, and the definition of $\mathcal M(t)$.

Similarly,
\[
\|u_j\cdot\nabla u_j\|_{L^2} \le \|u_j\|_{L^\infty}\|\nabla u_j\|_{L^2}, \quad \|a_2\nabla a_2\|_{L^2} \le \|a_2\|_{L^\infty}\|\nabla a_2\|_{L^2},
\]
which gives the second estimate in \eqref{vector-source-L1L2}. 

For the high derivatives, the Moser product estimate, \eqref{high-bootstrap}, and the definition of $\mathcal Z$ give
\begin{align*}
\|g_1\|_{H^8} &\le C\big(\|a_1\|_{W^{1,\infty}}\|u_1\|_{H^9} +\|u_1\|_{W^{1,\infty}}\|a_1\|_{H^9}\big),\\
\|g_2\|_{H^8} &\le C\|a_2\|_{L^\infty}\|a_2\|_{H^8} +C\|u_1\|_{W^{1,\infty}}\|u_1\|_{H^{10}},\\
\|g_3\|_{H^8} &\le C\big(\|u_2\|_{W^{1,\infty}}\|a_2\|_{H^9} +\|a_2\|_{W^{1,\infty}}\|u_2\|_{H^9}\big),\\
\|g_4\|_{H^8} &\le C\|a_2\|_{L^\infty}\|a_2\|_{H^8} +C\|u_2\|_{W^{1,\infty}}\|u_2\|_{H^{10}} +C\|a_2\|_{W^{1,\infty}}\|a_2\|_{H^{10}}.
\end{align*}
All high Sobolev norms on the right are bounded by $C\delta$ due to {$N\geq 10$}.  Hence
\[
\|F(\tau)\|_{H^8} \le C\delta\mathcal Z(\tau) \le C\delta\mathcal M(t)(1+\tau)^{-\frac{5}{4}},
\]
which is \eqref{source-H8}.
\end{proof}

We now combine Lemma \ref{lem-nonlinear-source} with the Green-function bounds in Lemma \ref{lem-green-estimates}. Let $U^\ell,U^m,U^h$ denote the low-, middle-, and high-frequency parts of the Duhamel formula corresponding to the regions introduced in Section \ref{S4}.

\begin{lem}\label{lem-nonlinear-decay}
Under \eqref{high-bootstrap}, one has
\begin{align*}
&(1+t)^{\frac{1}{2}}\|(a_1,a_2)(t)\|_{L^2} +(1+t)\|\nabla(a_1,a_2)(t)\|_{H^1} +(1+t)\|(q_1,q_2)(t)\|_{H^2} +(1+t)^{\frac{5}{4}}\mathcal Z(t)\\
&\quad \le C\varepsilon_N+C\delta\mathcal M(t)+C\mathcal M(t)^2,
\end{align*}
where $\varepsilon_N$ is the initial norm in \eqref{main-smallness}.
\end{lem}

\begin{proof}
We estimate the low-, middle-, and high-frequency parts separately.  The constants below are independent of $t$.

\medskip
\noindent\emph{Step 1: low-frequency estimates in $L^2$.} Using the divergence structure of $g_1,g_2$ and $g_4$ in
\eqref{nonlinear-g}, we transfer the outer spatial derivative to the
Green function.
By \eqref{source-L1} and \eqref{vector-source-L1L2},
\begin{align*}
\|F(\tau)\|_{L^1} +\|F_2^{\rm div}(\tau)\|_{L^1} +\|F_4^{\rm div}(\tau)\|_{L^1} \le C\big(\delta\mathcal M(t)+\mathcal M(t)^2\big)(1+\tau)^{-1} +C\varepsilon_Ne^{-\frac{\tau}{2}}.
\end{align*}
Moreover, \eqref{low-G-pointwise} gives the low-frequency orders
\[
\text{rows }1,3:\ (2,0,0,0),
\quad
\text{rows }2,4:\ (4,2,2,2).
\]
Hence Lemma \ref{lem-green-estimates} yields
\begin{align*}
\|(a_1^\ell,a_2^\ell)(t)\|_{L^2} &\le C(1+t)^{-\frac12}\varepsilon_N +C\big(\delta\mathcal M(t)+\mathcal M(t)^2\big) \int_0^t(1+t-\tau)^{-\frac34}(1+\tau)^{-1}\,d\tau\\
&\quad+C\varepsilon_N \int_0^t(1+t-\tau)^{-\frac34}e^{-\frac{\tau}{2}}\,d\tau\\
&\le C(1+t)^{-\frac12} \big(\varepsilon_N+\delta\mathcal M(t)+\mathcal M(t)^2\big),
\end{align*}
where
\[
\int_0^t(1+t-\tau)^{-\frac34}(1+\tau)^{-1}\,d\tau
\le C(1+t)^{-\frac12},
\quad
\int_0^t(1+t-\tau)^{-\frac34}e^{-\frac{\tau}{2}}\,d\tau
\le C(1+t)^{-\frac34}.
\]
For one and two spatial derivatives of the density components,
\begin{align*}
\int_0^t(1+t-\tau)^{-\frac54}(1+\tau)^{-1}\,d\tau  \le C(1+t)^{-1},\quad \int_0^t(1+t-\tau)^{-\frac74}(1+\tau)^{-1}\,d\tau \le C(1+t)^{-1}.
\end{align*}
Therefore,
\[
\|\nabla(a_1^\ell,a_2^\ell)(t)\|_{H^1}
\le C(1+t)^{-1}
\big(\varepsilon_N+\delta\mathcal M(t)+\mathcal M(t)^2\big).
\]
For $q_1^\ell,q_2^\ell$, the second and fourth rows contain at least two low-frequency powers, and
\[
\int_0^t(1+t-\tau)^{-\frac74}(1+\tau)^{-1}\,d\tau
\le C(1+t)^{-1}.
\]
Consequently,
\[
\|(q_1^\ell,q_2^\ell)(t)\|_{H^2}
\le C(1+t)^{-1}
\big(\varepsilon_N+\delta\mathcal M(t)+\mathcal M(t)^2\big).
\]

\medskip
\noindent\emph{Step 2: middle-frequency estimate.}
On the fixed annulus $r_0\le|\xi|\le R_0$, Lemma \ref{lem-green-estimates} gives an exponential semigroup bound.  Since all derivatives are equivalent on this annulus,
\begin{align}\label{middle-nonlinear}
\begin{aligned}
\|U^m(t)\|_{H^3} &\le Ce^{-\eta_2   t}\|U_0\|_{L^2} +C\int_0^te^{-\eta_2  (t-\tau)}\|F(\tau)\|_{L^2}\,d\tau \\
&\le C(1+t)^{-\frac{3}{2}} \big(\varepsilon_N+\delta\mathcal M(t)+\mathcal M(t)^2\big),
\end{aligned}
\end{align}
where \eqref{source-L2} was used in the last line.

\medskip
\noindent\emph{Step 3: high-frequency estimate.} Choose $\ell=5$ in \eqref{regularity-loss-ineq}.  From \eqref{high-G-pointwise}, for $0\le k\le3$,
\[ 
\left\|\nabla^k\int_0^tG^h(t-\tau)*F(\tau)\,d\tau\right\|_{L^2}
\le C\int_0^t(1+t-\tau)^{-\frac{5}{4}}\|F(\tau)\|_{H^{k+5}}\,d\tau.
\] 
Since $k+5\le8$, Lemma \ref{lem-nonlinear-source} yields
\[
\left\|\int_0^tG^h(t-\tau)*F(\tau)\,d\tau\right\|_{H^3} \le C\delta\mathcal M(t) \int_0^t(1+t-\tau)^{-\frac{5}{4}}(1+\tau)^{-\frac{5}{4}}\,d\tau \le C\delta\mathcal M(t)(1+t)^{-\frac{5}{4}}.
\]
The homogeneous high-frequency term is estimated by \eqref{high-linear-general}.  Since {$N\geq 10$} and $q_{i0}=\operatorname{div}u_{i0} (i=1,2)$,
\[
\|U_0\|_{H^8}\le C\varepsilon_N.
\]
Therefore
\begin{align}\label{high-nonlinear-final}
\|U^h(t)\|_{H^3}
\le C(1+t)^{-\frac{5}{4}}\big(\varepsilon_N+\delta\mathcal M(t)\big).
\end{align}

Combining Steps 1--3, we obtain
\begin{align}
\|(a_1,a_2)(t)\|_{L^2}
&\le C(1+t)^{-\frac12}\big(\varepsilon_N+\delta\mathcal M(t)+\mathcal M(t)^2\big),\label{T01}\\
\|\nabla(a_1,a_2)(t)\|_{H^1}
+\|(q_1,q_2)(t)\|_{H^2}
&\le C(1+t)^{-1}\big(\varepsilon_N+\delta\mathcal M(t)+\mathcal M(t)^2\big).\label{T02}
\end{align}

\medskip
\noindent\emph{Step 4: the $W^{1,\infty}$ norm.} We first estimate the low-frequency density components. By \eqref{low-P4}, the slow mode in the first and third rows contains two powers of $|\xi|$ on the initial datum $a_{10}$, but contains no additional low-frequency factor on $a_{20}$. On the other hand, $q_{10}=\operatorname{div}u_{10}$ and $q_{20}=\operatorname{div}u_{20}$ each provide one power of $|\xi|$ through \eqref{IBP-Green}. It follows from the low-frequency $L^2$--$L^\infty$ multiplier estimate that
\[
\left\| \bigl(G^\ell(t)*U_0\bigr)_{(1,3)} \right\|_{W^{1,\infty}} \le C\int_{|\xi|\le 2r_0} e^{-\eta_1 |\xi|^2t} |\widehat a_{20}(\xi)|\,d\xi +C\varepsilon_N(1+t)^{-\frac54}.
\]
Here $(\cdot)_{(1,3)}$ denotes the first and third components. Indeed, the contribution of $a_{10}$ decays at least as $(1+t)^{-7/4}$, while the contributions of $q_{10}$ and $q_{20}$ decay at least as $(1+t)^{-5/4}$. The exponentially stable low-frequency modes are also bounded by $C\varepsilon_N(1+t)^{-5/4}$.

We next estimate the Duhamel term. Applying the full Duhamel formula \eqref{nonlinear-duhamel}, and using \eqref{IBP-Green}, Bernstein's inequality, \eqref{low-kernel-L1inf}--\eqref{low-kernel-L2inf}, and \eqref{source-L1}--\eqref{source-L2}, we obtain
\begin{align*}
\|(a_1^\ell,a_2^\ell)(t)\|_{W^{1,\infty}} &\le C\int_{|\xi|\le 2r_0} e^{-\eta_1 |\xi|^2t} |\widehat a_{20}(\xi)|\,d\xi +C\varepsilon_N(1+t)^{-\frac54} 
+C\varepsilon_N \int_0^t (1+t-\tau)^{-\frac32}e^{-\frac{\tau}{2}}\,d\tau\\
&\quad +C\bigl(\delta\mathcal M(t)+\mathcal M(t)^2\bigr) \left[ \int_0^{\frac{t}{2}} (1+t-\tau)^{-\frac32}(1+\tau)^{-1}\,d\tau +\int_{\frac{t}{2}}^t (1+t-\tau)^{-\frac34}(1+\tau)^{-\frac32}\,d\tau \right].
\end{align*}
The negative Sobolev assumption on $a_{20}$ gives
\[
 \int_{|\xi|\le 2r_0} e^{-\eta_1 |\xi|^2t} |\widehat a_{20}(\xi)|\,d\xi  \le \left( \int_{|\xi|\le 2r_0} |\xi|^2e^{-2\eta_1 |\xi|^2t}\,d\xi \right)^{\frac12} \|a_{20}\|_{\dot H^{-1}} \le C(1+t)^{-\frac54} \|a_{20}\|_{\dot H^{-1}}.
\]
Moreover,
\begin{align*}
\int_0^t (1+t-\tau)^{-\frac32}e^{-\frac{\tau}{2}}\,d\tau &\le C(1+t)^{-\frac32},\\
\int_0^{\frac{t}{2}} (1+t-\tau)^{-\frac32}(1+\tau)^{-1}\,d\tau &\le C(1+t)^{-\frac54},\\
\int_{\frac{t}{2}}^t (1+t-\tau)^{-\frac34}(1+\tau)^{-\frac32}\,d\tau &\le C(1+t)^{-\frac54}.
\end{align*}
Combining these estimates yields
\[
\|(a_1^\ell,a_2^\ell)(t)\|_{W^{1,\infty}} \le C(1+t)^{-\frac54} \bigl( \varepsilon_N +\delta\mathcal M(t) +\mathcal M(t)^2 \bigr).
\]
 
Since $\operatorname{curl}u_1=0$, we get $u_1^\ell=-\nabla(-\Delta)^{-1}q_1^\ell$. Applying $\nabla(-\Delta)^{-1}$ directly to the second row of the full
Duhamel formula \eqref{nonlinear-duhamel}, the low-frequency orders
$(4,2,2,2)$ in \eqref{low-G-pointwise} become $(3,1,1,1)$.  Hence,
by \eqref{low-kernel-L1inf}--\eqref{low-kernel-L2inf},
\eqref{source-L1}--\eqref{source-L2}, and a splitting at $t/2$,
\begin{align*}
\|u_1^\ell(t)\|_{W^{1,\infty}}
&\le C(1+t)^{-\frac54}\varepsilon_N
+C\int_0^{\frac{t}{2}}(1+t-\tau)^{-2}\|F(\tau)\|_{L^1}\,d\tau\\
&\quad+C\int_{\frac{t}{2}}^t(1+t-\tau)^{-\frac54}\|F(\tau)\|_{L^2}\,d\tau
+C\varepsilon_Ne^{-\eta_1  t}
+C\int_0^te^{-\eta_1 (t-\tau)}\|F(\tau)\|_{L^2}\,d\tau\\
&\le C(1+t)^{-\frac54}
\big(\varepsilon_N+\delta\mathcal M(t)+\mathcal M(t)^2\big),
\end{align*}
where the exponentially stable part satisfies the same bound since
$\widehat\chi_\ell(\xi)|\xi|^{-1}\in L^2_\xi$.  Thus
\[
\|u_1^\ell(t)\|_{W^{1,\infty}}
\le C(1+t)^{-\frac54}
\big(\varepsilon_N+\delta\mathcal M(t)+\mathcal M(t)^2\big).
\]

For the second velocity, set $\omega_2=\operatorname{curl}u_2$. By the Helmholtz decomposition,
\[
u_2=-\nabla(-\Delta)^{-1}q_2
+\operatorname{curl}(-\Delta)^{-1}\omega_2.
\]
Since the low-frequency cut-off commutes with these Fourier multipliers,
taking the low-frequency part gives
\[
u_2^\ell=-\nabla(-\Delta)^{-1}q_2^\ell
+\operatorname{curl}(-\Delta)^{-1}\omega_2^\ell.
\]
For the first term, the fourth row of the low-frequency Green matrix
contains at least two powers of $|\xi|$, while
$\nabla(-\Delta)^{-1}$ contributes one factor $|\xi|^{-1}$.
Hence the same low-frequency kernel estimates as above yield
\[
\|\nabla(-\Delta)^{-1}q_2^\ell(t)\|_{W^{1,\infty}}
\le C(1+t)^{-\frac54}
\big(\varepsilon_N+\delta\mathcal M(t)+\mathcal M(t)^2\big).
\]
For the curl part, \eqref{omega2-decay} gives
\begin{align*}
\|\operatorname{curl}(-\Delta)^{-1}\omega_2^\ell(t)\|_{W^{1,\infty}}
&\le C\|\omega_2(t)\|_{H^2}
\le C\varepsilon_Ne^{-\frac{t}{2}}.
\end{align*}
Therefore,
\[
\|u_2^\ell(t)\|_{W^{1,\infty}}
\le C(1+t)^{-\frac54}
\big(\varepsilon_N+\delta\mathcal M(t)+\mathcal M(t)^2\big).
\]

On middle frequencies, the div--curl decomposition, \eqref{middle-nonlinear}, and
\eqref{omega2-decay} yield
\begin{align*}
 \|(a_1^m,a_2^m,u_1^m,u_2^m)(t)\|_{W^{1,\infty}}  \le C\|U^m(t)\|_{H^3}+C\|\omega_2^m(t)\|_{H^2} \le C(1+t)^{-\frac54}
\big(\varepsilon_N+\delta\mathcal M(t)+\mathcal M(t)^2\big).
\end{align*}
For the high-frequency part,
\begin{align*}
\|\nabla u_1^h\|_{H^2}\le C\|q_1^h\|_{H^2},\quad
\|\nabla u_2^h\|_{H^2}\le C\big(\|q_2^h\|_{H^2}+\|\omega_2^h\|_{H^2}\big),
\end{align*}
so \eqref{high-nonlinear-final}, \eqref{omega2-decay}, and
$H^3\hookrightarrow W^{1,\infty}$ imply
\[
\|(a_1^h,a_2^h,u_1^h,u_2^h)(t)\|_{W^{1,\infty}}
\le C(1+t)^{-\frac54}
\big(\varepsilon_N+\delta\mathcal M(t)+\mathcal M(t)^2\big).
\]
Combining the three frequency regions gives
\begin{align}\label{Z-decay}
\mathcal Z(t)
\le C(1+t)^{-\frac54}
\big(\varepsilon_N+\delta\mathcal M(t)+\mathcal M(t)^2\big).
\end{align}
Combining \eqref{T01}, \eqref{T02}, and \eqref{Z-decay} completes the proof.
\end{proof}

%
%
%
%
%
%

\section{Global existence and decay}\label{S6}

We complete the proof of Theorem \ref{thm-main} by combining the nonlinear decay estimate in Lemma \ref{lem-nonlinear-decay} with the high-order energy inequality in Lemma \ref{lem-high-order-energy}.
 
Let $T^*>0$ be the maximal existence time given by
Theorem \ref{thm-local}.  Fix $\delta_*>0$ sufficiently small so that all
estimates in the previous sections hold whenever
\begin{align}\label{section6-bootstrap}
\sup_{0\le \tau\le t}\mathcal E_N(\tau)^{\frac12}\le 2\delta_*.
\end{align}
Set
\[
T_1:=\sup\left\{T<T^*:
\sup_{0\le t\le T}\mathcal E_N(t)^{\frac12}\le 2\delta_*\right\}.
\]
For $\varepsilon_N$ sufficiently small, $T_1>0$ by Theorem \ref{thm-local} and the continuity of the solution.

For every $T<T_1$, Lemma \ref{lem-nonlinear-decay}, with $\delta=\delta_*$, yields
\begin{align}\label{M-close}
\mathcal M(T) \le C_0\varepsilon_N+C_0\delta_*\mathcal M(T) +C_0\mathcal M(T)^2,
\end{align}
where $C_0$ is independent of $T$.  Choose $\delta_*$ so that $C_0\delta_*\le1/8$ and then choose $\varepsilon_N$ sufficiently small. Since $\mathcal M(0)\le C\varepsilon_N$, a standard continuity argument applied to \eqref{M-close} gives the time-independent bound
\begin{align}\label{M-final}
\mathcal M(T)\le C\varepsilon_N, \quad 0\le T<T_1.
\end{align}
Consequently,
\begin{align}\label{final-decay}
\begin{aligned}
\|(a_1,a_2)(t)\|_{L^2} &\le C\varepsilon_N(1+t)^{-\frac12},\\
\|\nabla(a_1,a_2)(t)\|_{H^1} +\|(q_1,q_2)(t)\|_{H^2} &\le C\varepsilon_N(1+t)^{-1},\\
\mathcal Z(t) &\le C\varepsilon_N(1+t)^{-\frac54}, \quad 0\le t<T_1.
\end{aligned}
\end{align}
In particular, since $5/4>1$,
\[
\int_0^{T_1}\mathcal Z(\tau)\,d\tau\le C\varepsilon_N.
\]
Lemma \ref{lem-high-order-energy} and Gr\"onwall's inequality then imply
\begin{align}\label{EN-final}
\sup_{0\le t<T_1}\mathcal E_N(t) \le C\mathcal E_N(0) \exp\left(C\int_0^{T_1}\mathcal Z(\tau)\,d\tau\right) \le C\varepsilon_N^2e^{C\varepsilon_N} \le C\varepsilon_N^2.
\end{align}
Since $\varepsilon_0$ is sufficiently small, this yields
\[
\sup_{0\le t<T_1}\mathcal E_N(t)^{\frac12}\le\delta_*,
\]
which strictly improves \eqref{section6-bootstrap}.  Hence, by continuity, $T_1=T^*$.

Thus \eqref{M-final}--\eqref{EN-final} hold on $[0,T^*)$.  Sobolev embedding and the smallness of $\varepsilon_N$ give
\[
1+a_1(t,x)\ge\frac12, \quad 1+\frac{\gamma-1}{2c}a_2(t,x)\ge\frac12, \quad 0\le t<T^*.
\]
If $T^*<\infty$, the norm $\mathfrak X_N(t)$ remains uniformly bounded and both positivity conditions in the continuation criterion \eqref{continuation-criterion} remain uniformly strict.  Theorem \ref{thm-local} then allows the solution to be extended beyond $T^*$, a contradiction. Consequently $T^*=\infty$, and \eqref{final-decay} gives \eqref{main-decay-statement}.

%
%
%
%
%
%
\section*{Acknowledgments}
The work of the first author was supported by grants from the National Research Foundation of Korea (NRF), funded by the Korean government (MSIP) (Nos. 2022R1A2C1002820 and RS-2024-00406821). The work of the second author was supported by the National Natural Science Foundation of China (Grant No. 12501293, 12671258), the Anhui Provincial Natural Science Foundation (Grant No. 2408085QA031), and the Visiting Scholar Program of the Tianyuan Mathematics Research Center (Grant No. TY-2026-V-006). The work of the third author was supported by the China Scholarship Council (Grant No. 202406880015) and the National Natural Science Foundation of China (Grant No. 12001033).

%
%
%
%
%
%
\appendix
\section{Proof of local well-posedness} \label{appendix-local}

We prove Theorem \ref{thm-local} by a Friedrichs approximation at the regularity level $N\ge10$, preserving both the asymmetric energy structure and the negative Sobolev control of the charge.

\medskip
\noindent\emph{Step 1: the Poisson term and the negative norm.} Set
\[
\alpha=\frac{\gamma-1}{2}.
\]
Since $Q=\rho_1-\rho_2=b-\mathcal R(a_2)$ and $\Delta\Phi=Q$, the
Fourier multiplier representation gives
\begin{align}\label{local-poisson-est}
\|\nabla\Phi\|_{H^N}
=\|\nabla(-\Delta)^{-1}Q\|_{H^N}
\le C\bigl(\|Q\|_{\dot H^{-1}}+\|Q\|_{H^{N-1}}\bigr).
\end{align}
Indeed, the low-frequency part is controlled by
$\|Q\|_{\dot H^{-1}}$, while for $|\xi|\ge1$,
\[
(1+|\xi|^2)^N|\xi|^{-2}
\le C(1+|\xi|^2)^{N-1}.
\]

On a set where
\[
1+\frac{\gamma-1}{2c}a_2\ge\kappa>0,
\quad
\|a_2\|_{H^N}\le K,
\]
the map $a_2\mapsto\mathcal R(a_2)$ defined above is smooth. Hence the Moser composition estimate yields
\begin{align}\label{local-R-HN}
\|\mathcal R(a_2)\|_{H^{N-1}} \le C_{K,\kappa}\|a_2\|_{H^{N-1}}.
\end{align}
Moreover, $\mathcal R(0)=\mathcal R'(0)=0$, and hence
\[
|\mathcal R(a_2)| \le C_{K,\kappa}|a_2|^2.
\]
Since $L^{6/5}(\mathbb R^3)\hookrightarrow\dot H^{-1}(\mathbb R^3)$ and $H^1(\mathbb R^3)\hookrightarrow L^{12/5}(\mathbb R^3)$, we have
\begin{align}\label{local-R-negative}
\|\mathcal R(a_2)\|_{\dot H^{-1}}
&\le C\|\mathcal R(a_2)\|_{L^{\frac65}}
\le C_{K,\kappa}\|a_2\|_{L^{\frac{12}{5}}}^2
\le C_{K,\kappa}\|a_2\|_{H^1}^2.
\end{align}
It follows that
\begin{align}\label{local-Q-from-b}
\|Q\|_{\dot H^{-1}} \le \|b\|_{\dot H^{-1}} +C_{K,\kappa}\|a_2\|_{H^1}^2,\quad \|Q\|_{H^{N-1}} \le \|b\|_{H^{N-1}} +C_{K,\kappa}\|a_2\|_{H^{N-1}}.
\end{align}
Thus \eqref{local-poisson-est} is controlled on bounded positive sets by $\|b\|_{\dot H^{-1}}$, $\|b\|_{H^{N-1}}$, and $\|a_2\|_{H^N}$. 

\medskip
\noindent\emph{Step 2: Friedrichs approximation and the uniform high-order estimate.} For later use, set $\mathbf V=(b,u_1,a_2,u_2)$. For every integer $1\le \ell\le N$, define
\[
\|\mathbf V\|_{\mathscr H_\ell}^2 :=\|b\|_{\dot H^{-1}}^2 +\|b\|_{H^{\ell-1}}^2 +\|u_1\|_{H^\ell}^2 +\|a_2\|_{H^\ell}^2 +\|u_2\|_{H^\ell}^2.
\]

Let $J_m$ be a Friedrichs Fourier cutoff. We apply $J_m$ to each nonlinear product and composition in \eqref{main2}, and approximate the initial data by smooth data in the spaces of Theorem \ref{thm-local}. The density lower bounds are chosen to be at least $\kappa_0/2$. For fixed $m$, the cutoff makes the spatial derivative operators bounded on ${\rm Ran}\,J_m$, and the regularized vector field is locally Lipschitz in the corresponding Sobolev space. Therefore, the Banach-space Picard theorem gives a unique smooth approximate solution on a nontrivial interval. Since $J_m$ is self-adjoint and commutes with spatial derivatives, $\Lambda^s$, and the Riesz transforms, the cancellations used below remain valid for the approximate system.

For an $m$-independent lifespan, we estimate $a_1$ one derivative below $u_1$. Using $a_1=b+c^{-1}a_2$, the equations can be rewritten as
\begin{equation}\label{local-b-system}
\left\{
\begin{aligned}
&\partial_tb+\operatorname{div}(u_1-u_2)  +u_1\cdot\nabla b=H_b,\\
&\partial_tu_1+\mathscr R\Lambda^{-1}b  +u_1\cdot\nabla u_1  =\mathscr R\Lambda^{-1}\mathcal R(a_2),\\
&\partial_ta_2+c\operatorname{div}u_2  +u_2\cdot\nabla a_2  +\alpha a_2\operatorname{div}u_2=0,\\
&\partial_tu_2+c\nabla a_2-\mathscr R\Lambda^{-1}b+u_2  +u_2\cdot\nabla u_2+\alpha a_2\nabla a_2  =-\mathscr R\Lambda^{-1}\mathcal R(a_2),
\end{aligned}
\right.
\end{equation}
where $\mathscr R=\nabla\Lambda^{-1}$ and
\[
H_b= -\frac1c(u_1-u_2)\cdot\nabla a_2 -b\operatorname{div}u_1 -\frac1c a_2\operatorname{div}u_1 +\frac{\alpha}{c}a_2\operatorname{div}u_2.
\]

For $1\le s\le N$, apply $\Lambda^{s-1}$ to the first equation of
\eqref{local-b-system} and $\Lambda^s$ to the remaining three
equations. Taking the $L^2$ inner products with
$\Lambda^{s-1}b$, $\Lambda^su_1$, $\Lambda^sa_2$, and
$\Lambda^su_2$, respectively, gives two exact linear cancellations:
\begin{align*}
&c\langle\Lambda^s\operatorname{div}u_2,\Lambda^sa_2\rangle +c\langle\Lambda^s\nabla a_2,\Lambda^su_2\rangle=0,\notag\\
&\langle\mathscr R\Lambda^{s-1}b,  \Lambda^s(u_1-u_2)\rangle +\langle\mathscr R\!\cdot\!\Lambda^s(u_1-u_2),  \Lambda^{s-1}b\rangle=0.
\end{align*}
The first identity follows by integration by parts, and the second uses the skew-adjointness of the Riesz transforms.

At the highest order, the potentially derivative-losing terms are $a_2\operatorname{div}u_2$ and $a_2\nabla a_2$. Their principal parts cancel:
\[
 \int a_2\Lambda^s\operatorname{div}u_2\,\Lambda^sa_2\,dx +\int a_2\nabla\Lambda^sa_2\cdot\Lambda^su_2\,dx  =-\int\nabla a_2\cdot\Lambda^su_2\,\Lambda^sa_2\,dx.
\]
All remaining terms are commutators or lower-order products. We use
\begin{align*}
\|[\Lambda^s,f]g\|_{L^2} &\le C\bigl( \|\nabla f\|_{L^\infty}\|\Lambda^{s-1}g\|_{L^2} +\|\Lambda^sf\|_{L^2}\|g\|_{L^\infty} \bigr),\\
\|fg\|_{H^s} &\le C\bigl( \|f\|_{L^\infty}\|g\|_{H^s} +\|g\|_{L^\infty}\|f\|_{H^s} \bigr)
\end{align*}
to control all differentiated nonlinear terms.

At zero order, apply $\Lambda^{-1}$ to the first equation in \eqref{local-b-system}, pair it with $\Lambda^{-1}b$, and pair the remaining three equations with $u_1,a_2,u_2$, respectively. The Poisson terms cancel since
\[
\left\langle \Lambda^{-1}\operatorname{div}(u_1-u_2), \Lambda^{-1}b \right\rangle + \left\langle \mathscr R\Lambda^{-1}b,u_1-u_2 \right\rangle =0,
\]
while the acoustic terms cancel by integration by parts:
\[
c\langle\operatorname{div}u_2,a_2\rangle +c\langle\nabla a_2,u_2\rangle=0.
\]

At this level the nonlinear terms in the $b$ equation must be estimated in $\dot H^{-1}$ rather than merely in $L^2$. We use the Sobolev dual estimate $\|f\|_{\dot H^{-1}} \le C\|f\|_{L^{\frac65}}$. The nonlinear terms satisfy
\begin{align*}
\|u_1\cdot\nabla b\|_{\dot H^{-1}} &\le C\|u_1\|_{L^3}\|\nabla b\|_{L^2},\quad  \|(u_1-u_2)\cdot\nabla a_2\|_{\dot H^{-1}} \le C\|u_1-u_2\|_{L^3}\|\nabla a_2\|_{L^2},\\
\|b\operatorname{div}u_1\|_{\dot H^{-1}} &\le C\|b\|_{L^2}\|\nabla u_1\|_{L^3},\quad \|a_2\operatorname{div}u_j\|_{\dot H^{-1}} \le C\|a_2\|_{L^2}\|\nabla u_j\|_{L^3},
\quad j=1,2.
\end{align*}
Since $H^1(\mathbb R^3)\hookrightarrow L^3(\mathbb R^3)$ and all the positive Sobolev norms on the right are bounded on the local bootstrap set, these inequalities, together with \eqref{local-R-HN}--\eqref{local-Q-from-b}, control the pairing of the $\Lambda^{-1}$-equation with $\Lambda^{-1}b$. Therefore, on the same bounded positive set,
\begin{align}\label{local-zero-order-energy}
\begin{aligned}
&\frac{d}{dt}\Bigl( \|u_1\|_{L^2}^2+\|a_2\|_{L^2}^2+\|u_2\|_{L^2}^2 +\|b\|_{\dot H^{-1}}^2 \Bigr) +\|u_2\|_{L^2}^2\\
&\quad\le C_{K,\kappa}\Bigl( \|b\|_{\dot H^{-1}}^2+\|b\|_{H^{N-1}}^2 +\|u_1\|_{H^N}^2+\|a_2\|_{H^N}^2 +\|u_2\|_{H^N}^2 \Bigr).
\end{aligned}
\end{align}

Therefore, combining \eqref{local-zero-order-energy} with the estimates for $1\le s\le N$ gives, on every bootstrap set where the Sobolev norms are at most $K$ and the positive coefficients are bounded below by $\kappa$,
\begin{align}\label{local-high-energy}
\begin{aligned}
&\frac{d}{dt}\Bigl( \|b\|_{\dot H^{-1}}^2+\|b\|_{H^{N-1}}^2 +\|u_1\|_{H^N}^2+\|a_2\|_{H^N}^2 +\|u_2\|_{H^N}^2 \Bigr) +\|u_2\|_{H^N}^2\\
&\quad\le C_{K,\kappa}\Bigl( \|b\|_{\dot H^{-1}}^2+\|b\|_{H^{N-1}}^2 +\|u_1\|_{H^N}^2+\|a_2\|_{H^N}^2 +\|u_2\|_{H^N}^2 \Bigr).
\end{aligned}
\end{align}
The estimate uses only $b\in\dot H^{-1}\cap H^{N-1}$ and hence only $a_1\in H^{N-1}$ at positive Sobolev order.

Set $\widetilde Y_N(t):=\|\mathbf V(t)\|_{\mathscr H_N}$. Then \eqref{local-high-energy} reads
\begin{align}\label{local-YN-ineq} 
\frac{d}{dt}\widetilde Y_N^2(t) +\|u_2(t)\|_{H^N}^2 \le C_{K,\kappa}\widetilde Y_N^2(t).
\end{align}
Since $a_1=b+c^{-1}a_2$, the energy norm
$\|\mathbf V\|_{\mathscr H_N}$ is equivalent, up to fixed
constants, to the continuation norm in
\eqref{continuation-norm}. Moreover, \eqref{local-Q-from-b} recovers
the exact charge and the Poisson field from the components controlled
by this norm.

Let
\[
K=2C\bigl(1+\widetilde Y_N(0)\bigr).
\]
Gr\"onwall's inequality applied to \eqref{local-YN-ineq} gives a time $T_0=T_0(\widetilde Y_N(0),\kappa_0)>0$, independent of $m$, on which
\[
\widetilde Y_N(t)\le K.
\]
By taking $T_0$ smaller and using the time derivative bounds supplied by \eqref{local-b-system}, the two positive coefficients remain at least $\kappa_0/4$. Thus the approximate solutions exist uniformly on $[0,T_0]$ and satisfy bounds independent of $m$.

\medskip
\noindent\emph{Step 3: passage to the limit and time regularity.} The uniform estimates give weak-* compactness in the corresponding $L^\infty$-in-time Sobolev spaces. The equations also yield
\[
\partial_ta_1 \quad\hbox{bounded in }H^{N-2}, \quad \partial_t(u_1,a_2,u_2) \quad\hbox{bounded in }H^{N-1}.
\]
Local compactness and a diagonal argument on $\mathbb R^3$ give strong convergence in lower Sobolev norms, which is sufficient to pass to the nonlinear terms.

The limit solves \eqref{main2}. Since
\[
Q=b-\mathcal R(a_2),
\]
it also solves the original system. Therefore, it satisfies the exact charge equation
\begin{equation}\label{local-charge-equation}
\partial_tQ
+\operatorname{div}(\rho_1u_1-\rho_2u_2)=0.
\end{equation}

The compactness argument above initially gives only weak continuity at the top Sobolev indices. Strong continuity in the positive-order Sobolev spaces follows from the standard smoothing argument based on Friedrichs approximations; see \cite{BS75,K75}. Indeed, the uniform high-order estimate \eqref{local-YN-ineq}, together with the difference estimate obtained by repeating the same energy argument one derivative lower, allows one to apply this standard approximation argument. We obtain
\[
b\in C([0,T_0];H^{N-1}), \quad (u_1,a_2,u_2)\in C([0,T_0];H^N).
\]
Since $a_1=b+c^{-1}a_2$, it follows that
\[
a_1\in C([0,T_0];H^{N-1}).
\]
The equations and the Sobolev product estimates then yield
\[
\partial_ta_1\in C([0,T_0];H^{N-2}), \quad \partial_t(u_1,a_2,u_2) \in C([0,T_0];H^{N-1}).
\]

For completeness, the strong continuity of the negative-order component can also be verified directly from \eqref{local-charge-equation}:
\[
\Lambda^{-1}Q_t = -\Lambda^{-1}\operatorname{div} (\rho_1u_1-\rho_2u_2) \in C([0,T_0];L^2),
\]
since $\Lambda^{-1}\operatorname{div}$ is an order-zero multiplier and $\rho_ju_j\in C([0,T_0];L^2)$, $j=1,2$. Hence
\[
Q\in C([0,T_0];\dot H^{-1}).
\]
The difference version of \eqref{local-R-negative} shows that the composition map is locally Lipschitz from bounded positive subsets of $H^N$ into $\dot H^{-1}$. Since $a_2\in C([0,T_0];H^N)$, we have
\[
\mathcal R(a_2)
\in C([0,T_0];\dot H^{-1}).
\]
Recalling that $b=Q+\mathcal R(a_2)$, we conclude that
\[
b\in C([0,T_0];\dot H^{-1}).
\]
Estimate \eqref{local-poisson-est}, together with the positive-order continuity established above, gives
\[
\nabla\Phi\in C([0,T_0];H^N).
\]

Positivity follows directly from the continuity equations. If $X_j(t;x)$ denotes the flow generated by $u_j$, then
\[
\rho_j(t,X_j(t;x)) = \rho_{j0}(x) \exp\left( -\int_0^t \operatorname{div}u_j(\tau,X_j(\tau;x))\,d\tau \right), \quad j=1,2.
\]
Hence strictly positive initial densities remain strictly positive
as long as the classical solution exists.

\medskip
\noindent\emph{Step 4: uniqueness and continuous dependence.} Let $\mathbf V$ and $\widetilde{\mathbf V}$ be two solutions satisfying the above bounds on $[0,T_0]$, and set
\[
\delta\mathbf V = \mathbf V-\widetilde{\mathbf V} = (\delta b,\delta u_1,\delta a_2,\delta u_2).
\]
Subtracting their equations and applying the same energy method one derivative lower gives
\[
\frac{d}{dt}\mathcal D(t) \le C_{K,\kappa}\mathcal D(t), \quad \mathcal D(t) := \|\delta\mathbf V(t)\|_{\mathscr H_{N-1}}^2.
\]
The corresponding exact-charge difference is recovered from
\[
\delta Q = \delta b -\bigl(\mathcal R(a_2) -\mathcal R(\widetilde a_2)\bigr)
\]
by the locally Lipschitz composition estimate. If the initial data coincide, Gr\"onwall's inequality gives $\mathcal D(t)\equiv0$, and hence uniqueness.

Continuous dependence in the topology of \eqref{local-class} follows from the same smoothing argument based on Friedrichs truncations. Indeed, approximate both initial data sets by smooth truncations, use the uniform high-order estimate \eqref{local-YN-ineq} for the truncated solutions, and apply this lower-order stability estimate to compare two truncation levels. The high-frequency tails are controlled by the high-order estimate and the corresponding tails of the initial data. Passing first in the truncation parameter and then in the initial-data difference gives continuity of the solution map in the spaces stated in Theorem \ref{thm-local}.

\medskip
\noindent\emph{Step 5: continuation.} Assume that $T^*<\infty$ and that \eqref{continuation-criterion} holds. Then there exist $K<\infty$ and $\kappa>0$ such that, for every $t<T^*$,
\[
\mathfrak X_N(t)\le K, \quad 1+a_1(t,x)\ge\kappa, \quad 1+\frac{\gamma-1}{2c}a_2(t,x)\ge\kappa.
\]
The construction above shows that the local lifespan starting from any such time is bounded from below by a number $T_1=T_1(K,\kappa)>0$ independent of the starting time. Choose $t_0<T^*$ such that
\[
T^*-t_0<\frac{T_1}{2}
\]
and restart the Cauchy problem at $t=t_0$. The resulting local solution exists on $[t_0,t_0+T_1]$ and, by uniqueness, agrees with the original solution on their common interval. Since $t_0+T_1>T^*$, this extends the solution beyond its maximal existence time, a contradiction. This completes the proof of Theorem \ref{thm-local}.

%
%
%
%

\bibliographystyle{abbrv}
\bibliography{reference}

\end{document}